\documentclass[a4paper,12pt]{article}
\usepackage[utf8]{inputenc}
\usepackage{fullpage}
\usepackage[parfill]{parskip}
\usepackage{amsmath,amssymb,amsthm,stmaryrd,array,mathrsfs,mathtools}
\usepackage{tikz}
\usepackage{tikz-cd}
\usetikzlibrary{nfold}
\usepackage{xcolor}
\usepackage{enumitem}
\usepackage{adjustbox}
\usepackage{tipa}
\usepackage{ebproof}
\usepackage[unicode, colorlinks=true, linkcolor=blue, urlcolor=black, citecolor=blue]{hyperref}
\usepackage[all]{xy}
\usepackage{graphicx}
\usepackage{scalerel}[2016/12/29]
\usepackage{authblk}

\newcommand{\mast}[2]{\scaleobj{0.85}{\, \overset{\text{\raisebox{-0.25ex}{\smash{$#1$}}}}{\underset{\text{\raisebox{0.4ex}{\smash{$#2$}}}}{\ast}}} \, }
\newcommand{\dom}{\textup{dom}}
\newcommand{\cod}{\textup{cod}}

\newcommand{\vdashps}{\vdash_{\textup{ps}}}

\newcommand{\FV}{\textup{FV}}
\newcommand{\op}{\textup{\texttt{op}}}
\newcommand{\coh}{\textup{\texttt{coh}}}
\newcommand{\Ob}{\textup{\texttt{Ob}}}
\newcommand{\subto}[1]{\underset{\text{\raisebox{1ex}{\smash{\fontsize{5}{5}$#1$}}}}{\to}}

\newcommand{\up}{\textup{\texttt{u}}p}

\newcommand{\id}{\textup{\texttt{id}}}
\newcommand{\limit}{\textup{\texttt{lim}}}
\newcommand{\whisk}{\textup{\texttt{star}}}
\newcommand{\pstar}{\textup{\texttt{pstar}}}
\newcommand{\cone}{\textup{\texttt{cone}}}

\newcommand{\Hom}{\textup{Hom}}
\newcommand{\mcC}{\mathcal{C}}
\newcommand{\mcD}{\mathcal{D}}

\newcommand{\bN}{\mathbb{N}}

\newcommand{\CaTT}{\textup{\texttt{CaTT}}}

\newcommand{\sOb}{\textup{\footnotesize{(\texttt{Ob})}}}
\newcommand{\sTo}{\textup{\footnotesize{($\to$)}}}
\newcommand{\sVAR}{\textup{\footnotesize{(VAR)}}}
\newcommand{\sOP}{\textup{\footnotesize{(OP)}}}
\newcommand{\sCOH}{\textup{\footnotesize{(COH)}}}

\newcommand{\sES}{\textup{\footnotesize{(ES)}}}
\newcommand{\sSE}{\textup{\footnotesize{(SE)}}}

\tikzset{Rightarrow/.style={double equal sign distance,>={Implies},->},
triple/.style={-,preaction={draw,Rightarrow}},
quadruple/.style={preaction={draw,Rightarrow,shorten >=0pt},shorten >=1pt,-,double,double
distance=0.2pt}}

\newcommand{\multiquad}[1][1]{\hspace*{#1em}\ignorespaces} % For multiple quad after each other. Syntax example: \multquad[3]

\let\originalleft\left
\let\originalright\right
\renewcommand{\left}{\mathopen{}\mathclose\bgroup\originalleft}
\renewcommand{\right}{\aftergroup\egroup\originalright}

\let\originalgg\gg
\renewcommand{\gg}{\raisebox{1pt}{$\originalgg$}}

\newtheorem{theorem}{Theorem}[section]
\newtheorem{definition}[theorem]{Definition}

\newtheorem{lemma}[theorem]{Lemma}

\newtheorem{example}[theorem]{Example}
\newtheorem{remark}[theorem]{Remark}
\newtheorem{conjecture}[theorem]{Conjecture}

\let\oldexample=\example
\def\example{\oldexample\upshape}

\let\oldremark=\remark
\def\remark{\oldremark\upshape}

\title{A type-theoretic definition of lax $(\infty,\infty)$-limits}
\author{Thomas Jan Mikhail}
\affil{Universitat Autònoma de Barcelona\footnote{Email: \texttt{thomasjan.mikhail@uab.cat}}}
\date{}

\begin{document}

\maketitle

\begin{abstract}
    We introduce and study a purely syntactic notion of lax cones and $(\infty,\infty)$-limits on finite computads in \texttt{CaTT}, a type theory for $(\infty,\infty)$-categories due to Finster and Mimram. Conveniently, finite computads are precisely the contexts in \texttt{CaTT}. We define a cone over a context to be a context, which is obtained by induction over the list of variables of the underlying context. In the case where the underlying context is globular we give an explicit description of the cone and conjecture that an analogous description continues to hold also for general contexts. We use the cone to control the types of the term constructors for the universal cone. The implementation of the universal property follows a similar line of ideas. Starting with a cone as a context, a set of context extension rules produce a context with the shape of a transfor between cones, i.e.~a higher morphism between cones. As in the case of cones, we use this context as a template to control the types of the term constructor required for universal property.
\end{abstract}

\tableofcontents

%%%%%%%%%%%%%%%%%%%%%%%%%%%%%%%%%%%%%%%%%%%%%%%%%%%%%
%%%%%%%%%%%%%%%%%%%%%%%%%%%%%%%%%%%%%%%%%%%%%%%%%%%%%
\section{Introduction}
\label{sec:Limits}
%%%%%%%%%%%%%%%%%%%%%%%%%%%%%%%%%%%%%%%%%%%%%%%%%%%%%
%%%%%%%%%%%%%%%%%%%%%%%%%%%%%%%%%%%%%%%%%%%%%%%%%%%%%

With the advent of higher category theory, the 21st century is casting new light on abstract homotopy theory. By now the theory of $(\infty,1)$-categories has matured into a full fledged theory thanks to the pioneering work of André Joyal and Jacob Lurie among others, providing us with a language native to $\infty$-groupoids. 

Working with $(\infty,1)$-categories one quickly runs into $(\infty,2)$-categories, with the prototypical example of $(\infty,2)$-category being that of all $(\infty,1)$-categories. On the other hand, $(\infty,n)$-categories appear naturally in the study of cobordisms motivating the necessity for $(\infty,n)$-categories for arbitrary $n$. With these considerations in mind it thus becomes natural to consider the most general such structures namely $ (\infty,\infty)$-categories in which the existence of non-invertible morphisms is permitted in every dimension.

A number of different proposals for models for higher categories have been made. Approaches include Rezk's complete Segal $\Theta_n$-spaces~\cite{rezk2010cartesian} and Barwick's $n$-fold complete Segal spaces~\cite{barwick2005n}, as well as Verity's $n$-complicial sets (see Barwick--Schommer-Pries~\cite{barwick2021unicity} for a more extensive list and again Barwick--Schommer-Pries~\cite{barwick2021unicity} and Loubaton~\cite{loubaton2023theory} for the equivalence of the aforementioned examples). While non-algebraic models such as those mentioned have been the most successful so far, it is an important challenge to provide also a purely algebraic notion, as envisioned by Grothendieck. As an algebraic system, one may ask for a type-theoretic approach, which is the path taken in this work.

In the manuscript Pursuing Stacks, Grothendieck spelled out an algebraic definition of $\infty$-groupoids in terms of globular sets (see Maltsiniotis~\cite{maltsiniotis2010grothendieck}). This was then taken up by Maltsiniotis, who generalized this definition to $(\infty,\infty)$-categories~\cite{maltsiniotis2010grothendieck}. Brunerie~\cite{brunerie2016homotopy} developed a type-theoretic version of the definition of $ \infty$-groupoids with the aim of showing that the types in Homotopy Type Theory possess this structure. Building upon this and filling in the remaining slot, Finster and Mimram constructed a type-theoretic definition of $(\infty,\infty)$-categories called \texttt{CaTT}~\cite{finster2017type}, and together with Benjamin proved that its models are precisely the Grothendieck--Maltsiniotis $(\infty,\infty)$-categories~\cite{benjamin2021globular}.

With the definition in place, it is now our task to develop theory for it. One of the most basic categorical construction is that of limits. In higher dimensional categories there is an increased level of complexity, as the higher-dimensional cells comprising cones may be noninvertible. In 2-dimensional categories, for example, we may consider lax or oplax limits, depending on the orientation of the 2-cells involved, or pseudo limits if the 2-cells are invertible. Many fundamental concepts in 2-category theory arise as lax (co)limits. The Grothendieck construction corresponding to a pseudofunctor, for example, can be described as the oplax colimit of the given pseudofunctor. Furthermore, given a monad as a lax functor $ 1 \to \text{Cat} $, the Eilenberg--Moore category and the Kleisli category are computed by the lax limit and lax colimit of the lax functor respectively.

The same ideas are expected to continue to hold as we increase the dimension of categories all the way up to $(\infty,\infty)$-categories. The Grothendieck construction, for example, is again computed by the lax colimit in the setting of $ (\infty,\infty)$-categories, which motivates the study of lax limits in the setting of $ (\infty,\infty)$-categories (see Loubaton~\cite{loubaton2023theory} for a definition of lax (co)limits and the Grothendieck construction for $(\infty,\infty)$-categories, developed in a model independent framework). In this paper we propose a definition for lax $(\infty,\infty)$-limits in the type theory \texttt{CaTT} of Finster and Mimram~\cite{finster2017type}. In particular, we define a new theory $\CaTT{}_{\texttt{lim}}$ extending \CaTT{}, which describes $(\infty,\infty)$-categories with lax limits for finite computads. 

\texttt{CaTT} is a dependent type theory with two type constructors. The first, functioning as the base case, introduces the type $ \Ob $ which one may think of as the type of objects. The second one takes two terms $ s,t : A $ as input, and produces the type $ s \to_A t $, which is understood as the type of morphisms from $ s \to_A t $. Starting with the terms of $ \Ob $ and applying the second type constructor iteratively, one can access the type of all higher dimensional morphisms. In addition to the type constructors, \texttt{CaTT} also contains two term constructors. The first term constructor is responsible for the existence of all $(\infty,\infty)$-categorical operations. These include all binary operations. 
The second term constructor is responsible for all coherences, interpolating between the different ways of composing cells. These include the unit, the associator, the unit laws and so on. 

In \CaTT{}, diagrams can conveniently be encoded by contexts and we will use these words interchangeably. A context is a list of variables $ x_1 : A_1, \dots, x_n : A_n $, such that the type $ A_i $ can be constructed using the variables of $ x_1 : A_1, \dots, x_{i-1} : A_{i-1} $. Now, contexts in \CaTT{} are finite computads, thanks to a result by Benjamin, Markakis and Sarti~\cite{Benjamin_2024}. This finiteness constraint is brought upon us by the finiteness of type theory's contexts. It may be possible to translate and extend these ideas to other frameworks, such as that of Dean, Finster, Markakis, Reutter and Vicary~\cite{dean2024computads}, in which one is liberated from this restriction.

We define the cone over a context $ \Gamma $ to be another context $ K $. As a low dimensional example, consider the cone over $ \Gamma \equiv x : \Ob, y : \Ob, f : x \to y $, given by
\begin{center}
    \begin{tikzcd}
        c \\
        & x \\
        & y
        \arrow["{p_x}", from=1-1, to=2-2]
        \arrow[""{name=0, anchor=center, inner sep=0}, "{p_y}"', bend right=30, from=1-1, to=3-2]
        \arrow["f", from=2-2, to=3-2]
        \arrow["{p_f}"', shorten <=6pt, shorten >=4pt, Rightarrow, from=0, to=2-2]
    \end{tikzcd}
    \qquad\qquad
    \begin{tabular}{c}
        $ K \equiv \Gamma, \ c : \Ob, \ p_x : c \to x, \ p_y : c \to y, \ p_f : p_y \to p_x \mast{1}{0} f $.
    \end{tabular}
\end{center}

The definition of cones is based on the observation that the types of the projections exhibit a certain pattern. Take, for example, the variable $ f : x \to y $. First of all, the source of $ p_f $ is built out of the projections associated to the target of $ f $, namely $ p_y $. Second, the target of $ p_f $ is built out of the projections associated to the source of $ f $, namely $ p_x $, as well as $ f $ itself. Finally, the variable $ f $ appears in a certain linear way. These are the properties which appear as a set of side conditions in the recognition algorithm for cones below.

The recognition algorithm for cones makes use of an auxiliary judgment $ K \texttt{ cone } (\Gamma;c) $, subject to the rules
    \bigskip
    \begin{center}
        \begin{prooftree}
            \infer0{ c : \Ob \ \ \cone \ \ (\emptyset, c)}
        \end{prooftree}
    \end{center}
    \bigskip
    \begin{center}
        \begin{prooftree}
            \hypo{\Gamma, c : \Ob, \Pi \ \ \cone \ \ (\Gamma;c)}
            \hypo{\Gamma, x : X, c : \Ob, \Pi \vdash s \to_A t}
            \infer2{\Gamma, x : X, c : \Ob, \Pi, p_x : s \to_A t \ \ \cone \ \ ((\Gamma,x : X),c)}
        \end{prooftree}
        \qquad \footnotesize
        (+ side conditions) \normalsize
    \end{center}
    \bigskip
    If $ K \, \textup{\texttt{cone}} \, (\Gamma;c) $ is derivable we say that $ K $ is a cone over $ \Gamma $ with apex $ c $.

The rules exploit the inductive definition of contexts. For the empty context the cone is simply an apex. This is the first rule. In the second rule we begin with a cone $K$ over a diagram $ \Gamma $ as well as a context extension $ \Gamma, x : X $. The side conditions ensure that the type $ s \to_A t $ is of the appropriate form so as to be the type of a projection corresponding to the appended variable $ x : X $. Given this, the rule extends the original cone over $ \Gamma $ to a cone over $ \Gamma, x: X $.

The definition of the cone involves a choice of orientation for the higher cells. The orientation chosen has the benefit of exhibiting a certain uniformity. All other choices can be obtained by making suitable adjustments to the rules. Reversing the orientation of all 1-dimensional cells turns the definition into one for colimits.

If $ \Gamma $ is a globular diagram, meaning that all terms involved are variables, we show that there exists a context $ K $ such that $ K $ is the cone over $ \Gamma $. In particular we construct a cone over a globular diagram $ \Gamma $ with an explicit description of the type of the projection of a variable $ (x : A) \in \Gamma $:
\begin{equation} \label{eq:TypesOfProj}
    \begin{aligned}
    &c \to x, \qquad &&\text{if } \dim(x) = 0 \\
    &p_{\tau(x)} \to p_{\sigma(x)} \mast{d}{d-1} \Big(1_{p_{\sigma^2(x)}}^d \mast{d}{d-2} \cdots \Big(1_{p_{\sigma^{d}(x)}}^d \mast{d}{0} x \Big) \Big), &&\text{if } \dim(x) > 0.
    \end{aligned}
\end{equation}
where $ d = \dim(x) $. Here $ \mast{d}{k} $ denotes the binary composition of two $d$-dimensional cells along a $k$-dimensional gluing locus and $1^d_t$ denotes the $d$-dimensional iterative unit over the cell $t$, where $ d > \dim(t)$. We then examine two classes of examples of two dimensional non-globular diagrams, the first one containing sequences of composable 1-dimensional morphisms and the second a sequence of composable 2-dimensional morphisms. Considering both examples in the strict case we show that there exists a cone with a similar description to that in the globular case. The argument relies on the fact that, given a term in the diagram $ t : A $, the projections associated to the free variables of $ t $ may be composed in a way to obtain a certain term $ p_t $, the type of which is described by a formula analogous to equation~\ref{eq:TypesOfProj}. Restricting ourselves to the strict case does not spoil the argument, as the coherences are absorbed by the terms $ p_t $. Motivated by these examples we conjecture the existence of such terms for all $ t : A $ in a diagram. Using these we construct a cone with an explicit description for arbitrary diagrams.

Given a cone $ K $ over a diagram $ \Gamma $, we can build the universal cone with the help of term-constructor rules, which produce a term for the apex and for each projection. As a collection, these terms can be organized into a context morphism $ \Gamma \vdash \texttt{ucone} : K $. Diagrammatically we may depict this as
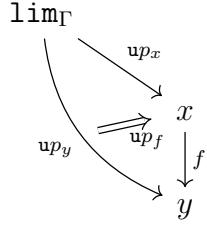
\begin{center}
    \begin{tikzcd}
        {\limit_\Gamma} \\
        & x \\
        & y
        \arrow["{\up_x}", from=1-1, to=2-2]
        \arrow[""{name=0, anchor=center, inner sep=0}, "{\up_y}"', bend right=30, from=1-1, to=3-2]
        \arrow["f", from=2-2, to=3-2]
        \arrow["{\up_f}"', shorten <=6pt, shorten >=4pt, Rightarrow, from=0, to=2-2]
    \end{tikzcd}
\end{center}

Implementing the universal property amounts to asking the functor of $(\infty,\infty)$-categories given by postcomposition with the universal cone, schematically depicted by
\begin{align} \label{eq:UniPropEquation}
    \texttt{cone}_* : \{\text{terms of } c \to \limit_\Gamma \} \longrightarrow \{\text{cones over } \Gamma \text{ with apex } c \}
\end{align}
to an equivalence. Here, the domain is the $(\infty,\infty)$-category of terms of the type $ c \to \limit_\Gamma $ and the codomain is given by the $(\infty,\infty)$-category of cones over $ \Gamma $ with apex $ c $. We refer to the $n$-dimensional cells of the codomain as $(n+1)$-transfors.\footnote{More generally, an $n$-dimensional cell in a functor category is called an $ n$-transfor, this terminology being coined by Crans~\cite{crans2003localizations}.} We define an equivalence of $ (\infty,\infty)$-categories to be a functor which is (essentially) surjective on all higher hom-$(\infty,\infty)$-categories. To ensure that the functor in equation~\ref{eq:UniPropEquation} is an equivalence we first spell out a set of rules which, in a manner similar to those generating cones, produce out of a given context a new context of the shape of a higher transfor between cones on that context. As in the case of the universal cone, we use this context as a template, to control the types of the terms we need to build with term constructor rules. In its first application, given a arbitrary cone over a given diagram, our constructions will produce a cone morphism (i.e.~a modification) from the given cone to the universal cone.

\textbf{Outline of paper:} We begin in Subsection~\ref{subsec:CaTT} with a short recap of the rules making up the type theory \texttt{CaTT}. In Subsection~\ref{subsec:Operations} we take the time to make explicit how the standard binary operations and identities are extracted from the general rules.

In Section~\ref{sec:UniversalCone} we discuss all the constructions related to the universal cone. After motivating the definition we introduce the rules generating the cone over a given diagram in Subsection~\ref{subsec:ConesAsContext}. In Subsection~\ref{subsec:ConesGlobCont} we discuss the existence of cones over globular as well as arbitrary contexts. Subsection~\ref{subsec:UniCone} introduces the rules for the universal cone.

Section~\ref{sec:UniProp} contains all constructions related to the universal property. As a warm-up we begin in Subsection~\ref{subsec:Gray} with a set of rules which, given an arbitrary diagram $\Gamma$ generate contexts which have the shape of the Gray tensor product of $ \Gamma $ with the interval. In fact the same rules also produce Gray tensor products with the $n$-globe for any $ n \in \bN^> $. We give an existence proof for the case $n=1$. In Subsection~\ref{subsec:HiTrsfCones} we then give a slight modification of the previous rules so as to give higher transfors between cones. The last ingredient we need is the ability to postcompose with a cone, as in equation~\ref{eq:UniPropEquation}. This is the subject of Subsections~\ref{subsec:Whiskering} and~\ref{subsec:WhiskUP}. Finally, putting everything together we give in Subsection~\ref{subsec:UP} a rule which generates all the required terms for the universal cone to satisfy the universal property. Subsection~\ref{subsec:FreeVars} covers some technical loose ends, ensuring that the new rules do not interfere with those of \texttt{CaTT} and that the new rules do not spoil the admissibility of the cut rule.

\textbf{Acknowledgments:} I am deeply indebted to Joachim Kock for his constant support and guidance. My gratitude extends also to Michael Shulman who first suggested looking into \texttt{CaTT} and the conversations with whom have been a major boost to my project. Many more have given me their time and feedback as I was developing the ideas for this project. These include Thibaut Benjamin, Simon Henry, Felix Loubaton, Samuel Mimram and Chaitanya Leena Subramaniam. Benjamin in particular has made a number of helpful suggestions. This work has been funded by the grant FI-DGR 2020 of the Ag\`encia de Gesti\'o d’Ajuts Universitaris i de Recerca of Catalunya, Spain. I also acknowledge support from grant PID2020-116481GB-I00 (AEI/FEDER, UE) of Spain, grant 10.46540/3103-00099B from the Independent Research Fund Denmark, and the Danish National Research Foundation through the Copenhagen Centre for Geometry and Topology (DNRF151) 

%%%%%%%%%%%%%%%%%%%%%%%%%%%%%%%%%%%%%%%%%%%%%%%%%%%%%%%%%%%
%%%%%%%%%%%%%%%%%%%%%%%%%%%%%%%%%%%%%%%%%%%%%%%%%%%%%%%%%%%
\section{Type theory for \texorpdfstring{$(\infty,\infty)$}{(oo,oo)}-categories}
%%%%%%%%%%%%%%%%%%%%%%%%%%%%%%%%%%%%%%%%%%%%%%%%%%%%%%%%%%%
%%%%%%%%%%%%%%%%%%%%%%%%%%%%%%%%%%%%%%%%%%%%%%%%%%%%%%%%%%%

%%%%%%%%%%%%%%%%%%%%%%
%%%%%%%%%%%%%%%%%%%%%%
\subsection{The type theory \CaTT{}} \label{subsec:CaTT}
%%%%%%%%%%%%%%%%%%%%%%
%%%%%%%%%%%%%%%%%%%%%%

The type theory \CaTT{}, developed by Mimram and Finster~\cite{finster2017type} describes a single $(\infty,\infty)$-category. Its models are precisely the $(\infty,\infty)$-categories in the sense of Grothendieck and Maltsiniotis, as shown by Benjamin, Finster and Mimram~\cite{benjamin2021globular}. The type theory \CaTT{} has two type constructors
\begin{center}
    \begin{prooftree}
        \infer0{\Gamma \vdash \Ob}
    \end{prooftree}
    \qquad\qquad
    \begin{prooftree}
        \hypo{\Gamma \vdash s : A}
        \hypo{\Gamma \vdash t : A}
        \infer2{\Gamma \vdash s \to_A t}
    \end{prooftree}
\end{center}
Here $ \Ob $ is the type of all objects, while $ s \to_A t $ may be thought of as a ``directed hom-type'', containing all morphisms from $ s $ to $ t $. Starting with $ \Ob $, the second rule allows us to iteratively build all higher hom-types. To reduce clutter we will often omit the subscript $ A $ in $ s \to_A t $.

In a (weak) higher category there are two types of cells one needs to produce as part of the axioms, the operations and the coherences. Operations are non-invertible and include precisely all the compositions while the coherences are invertible and include cells such as the units, the associators and the interchange laws. As an example, consider the horizontal composition of two 2-dimensional cells (left) and the associator (right):
\begin{equation} \label{diag:HCompAssoc}
    \begin{tikzcd}
        x & y & z
        \arrow[""{name=0, anchor=center, inner sep=0}, "{f_1}", bend left=40, from=1-1, to=1-2]
        \arrow[""{name=1, anchor=center, inner sep=0}, "{g_1}"', bend right=40, from=1-1, to=1-2]
        \arrow[""{name=2, anchor=center, inner sep=0}, "{s \equiv f_1 \cdot f_2}", bend left=80, from=1-1, to=1-3, dashed]
        \arrow[""{name=3, anchor=center, inner sep=0}, "{t \equiv g_1 \cdot g_2}"', bend right=80, from=1-1, to=1-3, dashed]
        \arrow[""{name=4, anchor=center, inner sep=0}, "{f_2}", bend left=40, from=1-2, to=1-3]
        \arrow[""{name=5, anchor=center, inner sep=0}, "{g_2}"', bend right=40, from=1-2, to=1-3]
        \arrow["\alpha", shorten <=3pt, shorten >=3pt, Rightarrow, from=0, to=1]
        \arrow["\op"'{pos=0.7}, shift left=3, shorten <=8pt, shorten >=8pt, Rightarrow, from=2, to=3, dashed]
        \arrow["\beta", shorten <=3pt, shorten >=3pt, Rightarrow, from=4, to=5]
    \end{tikzcd}
    \qquad\qquad
    \begin{tikzcd}
        x & y & z & w
        \arrow["f", from=1-1, to=1-2]
        \arrow["{f\cdot g}", bend left=50, from=1-1, to=1-3, dashed]
        \arrow[""{name=0, anchor=center, inner sep=0}, "{s \equiv (f\cdot g)\cdot h}", bend left=70, from=1-1, to=1-4, dashed]
        \arrow[""{name=1, anchor=center, inner sep=0}, "{t \equiv f\cdot (g\cdot h)}"', bend right=70, from=1-1, to=1-4, dashed]
        \arrow["g", from=1-2, to=1-3]
        \arrow["{g\cdot h}", bend right=50, from=1-2, to=1-4, dashed]
        \arrow["h", from=1-3, to=1-4]
        \arrow["\coh"'{pos=0.6}, shift left=2.45, shorten <=8pt, shorten >=8pt, Rightarrow, from=0, to=1, dashed]
    \end{tikzcd}
\end{equation}
The collection of solid arrows form the pasting diagram. These are the diagrams which may be composed via the operations. The remaining arrows, stylized with a dashed body, are built using the data of the underlying pasting diagram. In each case we are ultimately constructing a 2-dimensional cell $ s \to t $. In the case of the horizontal composition $ \op : f_1 \cdot f_2 \to g_1 \cdot g_2 $, the source only makes use of variables $ f_1 $ and $ f_2 $ which define a sub-pasting-diagram called the source and denoted by $ \partial^-\Gamma $. Similarly, the target of $ \op $ only makes use of $ g_1 $ and $ g_2 $ which define a sub-pasting-diagram called the target diagram and denoted by $ \partial^+\Gamma $. In the case of the associator $ \coh : (f\cdot g) \cdot h \to f \cdot (g \cdot h) $, both the source and the target make use of the whole diagram defined by $ f,g $ and $ h $. These observations are the defining features of operations and coherences respectively.

In the type theory \CaTT{} the underlying pasting diagram is encoded as a context $ \Gamma $. A judgment of the form $ \Gamma \vdashps $ asserts that $ \Gamma $ as a diagram has the shape of a pasting diagram. \CaTT{} contains two rules (OP) and (COH), producing the operations and coherences respectively. Each of these two rules is accompanied by a side condition referencing the free variables of the terms involved, which expresses precisely the intuition explained in the previous paragraph. The rules read

\bigskip

\begin{center}
    \begin{prooftree}
        \hypo{\Gamma \vdashps}
        \hypo{\Gamma \vdash s \to_A t}
        \infer2{\Gamma \vdash \op_{\Gamma,s \to_A t} : s \to_{A} t}
    \end{prooftree}
    \qquad
	\begin{tabular}{c}
		$ \FV(s : A) = \FV(\partial^- \Gamma) $ \\[1mm]
		$ \FV(t : A) = \FV(\partial^+\Gamma) $
	\end{tabular}
\end{center}
\bigskip
\begin{center}
    \begin{prooftree}
        \hypo{\Gamma \vdashps}
        \hypo{\Gamma \vdash s \to_A t}
        \infer2{\Gamma \vdash \coh_{\Gamma,s \to_A t} : s \to_{A} t}
    \end{prooftree}
    \qquad
	\begin{tabular}{c}
		$ \FV(s : A) = \FV(\Gamma) $ \\[1mm]
		$ \FV(t : A) = \FV(\Gamma) $
	\end{tabular}
\end{center}

\medskip

It is now just a matter of building in a cut into the rules to make the cut rule admissible. Intuitively we may think of this as allowing us to compose arbitrary terms, not just the variables of the given pasting diagram. In the next section we will give the full list of all rules with all the required details.

\begin{definition}[\CaTT{} Rules] \label{def:Rules} \CaTT{} is defined by the rules:
	
	Rules for types:
	
	\begin{center}
		\begin{prooftree}
			\hypo{\Gamma \vdash}
			\infer1[\sOb]{\Gamma \vdash \Ob}
		\end{prooftree}
		\qquad\qquad
		\begin{prooftree}
			\hypo{\Gamma \vdash s : A}
			\hypo{\Gamma \vdash t : A}
			\infer2[\sTo]{{\Gamma \vdash s \to_{A} t}}
		\end{prooftree}
	\end{center}
	
	Rules for terms:
	
	\begin{center}
		\begin{prooftree}
			\hypo{\Gamma \vdash}
			\hypo{(x : A) \in \Gamma}
			\infer2[\sVAR]{\Gamma \vdash x : A}
		\end{prooftree}
	\end{center}
	\vspace{5mm}
	\begin{center}        
		\begin{prooftree}
			\hypo{\Gamma \vdashps}
			% \hypo{\partial^- \Gamma \vdash s : A}
			% \hypo{\partial^+ \Gamma \vdash t : A}
            \hypo{\Gamma \vdash s \to_A t}
			\hypo{\Delta \vdash \gamma : \Gamma}
			\infer3[\sOP]{\Delta \vdash \op_{\Gamma, s\to_{A} t}[\gamma] : s[\gamma] \to_{A[\gamma]} t[\gamma] }
		\end{prooftree}
		\quad
		\begin{tabular}{ c }
			$ \FV(\partial^- \Gamma) = \FV(s : A) $ \\[2mm]
			$ \FV(\partial^+ \Gamma) = \FV(t : A) $  
		\end{tabular}
	\end{center}
	\vspace{5mm}
	\begin{center}
		\begin{prooftree}
			\hypo{\Gamma \vdashps}
			\hypo{\Gamma \vdash s \to_{A} t}
			\hypo{\Delta \vdash \gamma : \Gamma}
			\infer3[\sCOH]{\Delta \vdash \coh_{\Gamma, s\to_{A} t}[\gamma] : s[\gamma] \to_{A[\gamma]} t[\gamma]}
		\end{prooftree}
		\qquad
		\begin{tabular}{ c }
			$ \FV(\Gamma) = \FV(t : A) $ \\[2mm]
			$ \FV(\Gamma) = \FV(s : A) $  
		\end{tabular}
	\end{center}

    Rules for contexts:
	
	\begin{center}
		\begin{prooftree}
			\infer0[\footnotesize{\textup{(EC)}}]{\emptyset \vdash}
		\end{prooftree}
		\qquad\qquad
		\begin{prooftree}
			\hypo{\Gamma \vdash A}
			\infer1[\footnotesize{\textup{(CE)}}]{\Gamma, x : A \vdash}
		\end{prooftree}
	\end{center}
	
	Rules for substitutions:
	
	\begin{center}
		\begin{prooftree}
			\hypo{\Delta \vdash}
			\infer1[\sES]{\Delta \vdash \langle \,\rangle : \emptyset}
		\end{prooftree}
		\qquad\qquad
		\begin{prooftree}
			\hypo{\Delta \vdash \gamma : \Gamma}
			\hypo{\Gamma, x : A \vdash}
			\hypo{\Delta \vdash t : A[\gamma]}
			\infer3[\sSE]{\Delta \vdash \langle \gamma, t \rangle : (\Gamma, x : A)}
		\end{prooftree}
	\end{center}
	
	Rules for ps-contexts:
	
	\begin{center}
		\begin{prooftree}
			\infer0[\footnotesize{\textup{(PSS)}}]{x : \Ob \vdashps x : \Ob}
		\end{prooftree}
		\qquad\qquad
		\begin{prooftree}
			\hypo{\Gamma \vdashps x : A}
			\infer1[\footnotesize{\textup{(PSE)}}]{\Gamma, y : A, f : x \to_{A} y \vdashps f : x \to_A y}
		\end{prooftree}
	\end{center}
	\vspace{3mm}
	\begin{center}
		\begin{prooftree}
			\hypo{\Gamma \vdashps f : x \to_{A} y}
			\infer1[\footnotesize{\textup{(PSD)}}]{\Gamma \vdashps y : A}
		\end{prooftree}
		\qquad\qquad
		\begin{prooftree}
			\hypo{\Gamma \vdashps x : \Ob}
			\infer1[\footnotesize{\textup{(PS)}}]{\Gamma \vdashps}
		\end{prooftree}
	\end{center}
	\bigskip
\end{definition}

Given a context morphism $ \Delta \vdash \gamma : \Gamma $ and a variable $ x \in \FV(\Gamma) $ we may denote by $ \gamma_x $ the term in $ \gamma $ corresponding to $ x $.

To make sense of the rules we need to define the free variables on all the constructors, substitution on the constructors, the dimension of contexts and finally the source and the target of a ps-context.

\begin{definition}[Free Variables]
    The set of free variables are inductively defined as follows:
    \begin{align*}
        \FV(\Ob) &:= \emptyset & \FV(\emptyset) &:= \emptyset \\
        \FV(s \to_A t) &:= \FV(s)\cup\FV(t)\cup\FV(A) & \FV(\Gamma, x : A) &:= \FV(\Gamma)\cup \{x\} \\
        \\
        \FV(x) &:= \{x\} & \FV(\langle \, \rangle) &:= \emptyset \\
        \FV(\op_{\Gamma, s \to_A t}[\gamma]) &:= \FV(\gamma) & \FV(\langle \gamma, t \rangle) &:= \FV(\gamma)\cup\FV(t) \\
        \FV(\coh_{\Gamma, s \to_A t}[\gamma]) &:= \FV(\gamma).
    \end{align*}
\end{definition}
As a shorthand we also write $ \FV(t : A) = \FV(t)\cup\FV(A) $ for a term $ t $ of type $ A $. Moreover, if a context $ \Gamma $ is of the form $ \Gamma', \Gamma'' $, then we define $ \FV(\Gamma'') = \FV(\Gamma)\backslash\FV(\Gamma') $. The cardinality of $ \FV(\Gamma'') $ will be denoted by $ |\Gamma''| $.

For technical reasons it is convenient to make substitution an admissible rule rather than an explicit one. This requires defining substitution on the constructors and building in just enough substitution into the term constructors. %We give the full definitions here but will refrain from any further discussion of this rather technical subject.
\begin{definition}[Substitution] \label{def:Substitution}
	For the operation of substitution, we define
	\begin{align*}
		\Ob[\gamma] &:\equiv \Ob & \langle \, \rangle \circ \gamma &:\equiv \langle \, \rangle \\
		(s \subto{A} t)[\gamma] &:\equiv s[\gamma] \subto{A[\gamma]} t[\gamma] & \langle \theta, \, t \rangle \circ \gamma &:\equiv \langle \theta \circ \gamma, \, t[\gamma] \rangle \\
	    \\
        x_i[\gamma] &:\equiv \gamma_i \\
		\op_{\Gamma, s \to t}[\gamma][\delta ] &:\equiv \op_{\Gamma, s \to t}[\gamma \circ \delta] \\
        \coh_{\Gamma, s \to t}[\gamma][\delta ] &:\equiv \coh_{\Gamma, s \to t}[\gamma \circ \delta].
	\end{align*}
	Moreover, if $ (x_i : A_i) \in \Gamma $ and $ \Delta \vdash \gamma : \Gamma $, then we define $ A_i[\gamma] :\equiv A_i[\gamma^{i-1}] $.
\end{definition}

The rules also rely on the definition of the source and target of a ps-context. This definition makes use of the dimension of contexts which we define first.

\begin{definition}[Dimension]
    The dimension of a type is defined inductively by:
    \begin{align*}
        \dim(\Ob) &:= 0 \\
        \dim(s \to_A t) &:= \dim(A) + 1.
    \end{align*}
    Moreover, given a context $ \Gamma \equiv (x_i : A_i)_{1 \le i \le n} $ we define $ \dim(\Gamma) := \max_i\{\dim(A_i)\} $ and given a term $ \Gamma \vdash t : A $ we define $ \dim(t) := \dim(A) $.
\end{definition}

Given a term $ \Gamma \vdash t : A $ with $ \dim(t) > 0 $, then $ A $ is necessarily the result of an application of ($\to \mathcal{I}$), and we write $ A \equiv \sigma(t) \to_{\partial A} \tau(t) $.

\begin{definition}[Context Source and Target] \label{def:PsContextSourceAndTarget}
	Let $ \Gamma \vdashps $ be a ps-context and let $ k \in \bN $ be a natural number. We define $ k$-source by $ \partial_k^-( x : \Ob) :\equiv x : \Ob $ and
	\begin{align*}
		\partial_k^- (\Gamma, y : A, f : x \to y) :\equiv
		\begin{cases}
			\partial_k^- (\Gamma) & \text{if } k \le \dim(A) \\
			\partial_k^- (\Gamma), \ y : A, \ f : x \to y & \text{else}
		\end{cases}
	\end{align*}
	and the $k$-th target by $ \partial_k^+(x : \Ob) :\equiv x : \Ob $ and 
	\begin{align*}
		\partial_k^+(\Gamma, \ y : A, \ f : x \to y) :\equiv
		\begin{cases}
			\partial_k^+(\Gamma) & \text{if } k < \dim(A) \\
			\textup{drop}\bigl( \partial_k^+(\Gamma) \bigr), \ y : A & \text{if } k = \dim(A) \\
			\partial_k^+(\Gamma), \ y : A, \ f : x \to y & \text{else}
		\end{cases}
	\end{align*}
	Here $ \textup{drop}(\Gamma) $ is given by $ \Gamma $ with the last element removed. If $ \dim(\Gamma) > 0 $, then the source and target of $ \Gamma $ are defined to be $ \partial^-\Gamma :\equiv \partial_{\dim(\Gamma) -1}^- \Gamma $ and $ \partial^+\Gamma :\equiv \partial_{\dim(\Gamma)-1}^+\Gamma $.
\end{definition}

If we forget the rules (OP) and (COH) along with all the required rules (such as those for the ps-contexts and their boundaries), we are left with a type theory called \texttt{GSet}, which has been studied by Benjamin, Finster and Mimram in~\cite{benjamin2021globular}. The models of this type theory are precisely globular sets. A term, type or context in \CaTT{} is said to be globular, if it is derivable in the subtype theory \texttt{GSet}. This amounts to saying that all of the terms appearing in its construction are variables. Similarly a term, type or context is said to be categorical if it is constructed purely using the rules of \CaTT{}.

\begin{lemma} The following two statements hold in \textup{\texttt{CaTT}}
	\begin{enumerate}[label=(\roman*)]
		\item Given a judgment $ \Gamma \vdash A $ we have 
		\begin{equation*}
			x_i \in \FV(A) \qquad \Longrightarrow \qquad \dim(A_i) < \dim(A).
		\end{equation*}
		\item Given a judgment $ \Gamma \vdash t : A $ we have 
		\begin{align*}
			x_i \in \FV(t) \qquad \Longrightarrow \qquad \dim(x_i) \le \dim(t).
		\end{align*}
	\end{enumerate}
	\label{lem:FreeVariablesBoundTypeDim}
\end{lemma}

\subsection{Operations and identities} \label{subsec:Operations}
%%%%%%%%%%%%%%%%%%%%%%%%%%%%%%%%%%%%%%%%%%%%%%%%%
%%%%%%%%%%%%%%%%%%%%%%%%%%%%%%%%%%%%%%%%%%%%%%%%%

In a higher category, there are various ways of composing higher dimensional cells. Here we make explicit the $d$ canonical binary compositions with which $d$-dimensional cells may be composed. The $i$-th binary composition is defined by having an $(i-1)$-dimensional locus of composition.

For 2-dimensional cells we have vertical and horizontal composition. The relevant diagram for vertical composition is given by
\begin{center}
    \begin{tikzcd}[column sep=large]
        x & y
        \arrow[""{name=0, anchor=center, inner sep=0}, "g"{description}, from=1-1, to=1-2]
        \arrow[""{name=1, anchor=center, inner sep=0}, "f", bend left=60, from=1-1, to=1-2]
        \arrow[""{name=2, anchor=center, inner sep=0}, "h"'{pos=0.55}, bend right=50, from=1-1, to=1-2]
        \arrow["\alpha", shorten <=4pt, shorten >=4pt, Rightarrow, from=1, to=0]
        \arrow["\beta", near start, shorten <=5pt, shorten >=3pt, Rightarrow, from=0, to=2]
    \end{tikzcd}
\end{center}
As a ps-context it can be obtained by the derivation tree
\bigskip
\begin{center}
    \begin{prooftree}
        \infer0[\footnotesize{\textup{(PSS)}}]{x : \Ob \vdashps x : \Ob}
        \infer1[\footnotesize{\textup{(PSE)}}]{x : \Ob, \, y : \Ob, \, f : x \to y \vdashps f : x \to y}
        \infer1[\footnotesize{\textup{(PSE)}}]{x : \Ob, \, y : \Ob, \, f : x \to y, \, g : x \to y, \, \alpha : f \to g \vdashps \alpha : f \to g }
        \infer1[\footnotesize{\textup{(PSD)}}]{x : \Ob, \, y : \Ob, \, f : x \to y, \, g : x \to y, \, \alpha : f \to g \vdashps g : x \to y }
        \infer1[\footnotesize{\textup{(PSE)}}]{x : \Ob, \, y : \Ob, \, f : x \to y, \, g : x \to y, \, \alpha : f \to g, \, h : x \to y, \, \beta : g \to h \vdashps \beta : g \to h}
    \end{prooftree}
\end{center}
\bigskip
and we define the context in the last judgment to be $ O^2_1(\alpha,\beta) $, that is we set
\begin{align*}
    O^2_1(\alpha,\beta) :\equiv x : \Ob, y : \Ob, f : x \to y, g : x \to y, \alpha : f \to g, h : x \to y, \beta : g \to h.
\end{align*}
The source and boundary ps-contexts are given by
\begin{align*}
    \partial^- O^2_1(\alpha,\beta) \equiv x : \Ob, y : \Ob, f : x \to y \equiv:& \, D^1(f) \\[2mm]
    \partial^+ O^2_1(\alpha,\beta) \equiv x : \Ob, y : \Ob, h : x \to y \equiv:& \, D^1(h).
\end{align*}
With these we may apply the (OP) rule and obtain the vertical composition
\begin{center}
    \begin{prooftree}
        \hypo{O^2_1(\alpha,\beta) \vdashps }
        \hypo{D^1(f) \vdash f : x \to z}
        \hypo{D^1(h) \vdash h : x \to z}
        \hypo{\Delta \vdash \gamma : O^2_1(\alpha,\beta)}
        \infer4[\sOP]{\Delta \vdash \gamma_\alpha \mast{2}{1} \gamma_\beta : \gamma_f \to \gamma_h }
    \end{prooftree}
\end{center}
where we have defined $ \gamma_\alpha \mast{2}{1} \gamma_\beta :\equiv \op_{O^2_1(\alpha,\beta)}[\gamma] $. In particular we get the composite $ \alpha \mast{2}{1} \beta : f \to h $ using the identity context morphism.

Similarly we can build horizontal composition by considering the diagram
\begin{center}
    \begin{tikzcd}
        x & y & z
        \arrow[""{name=0, anchor=center, inner sep=0}, "f", bend left=40, from=1-1, to=1-2]
        \arrow[""{name=1, anchor=center, inner sep=0}, "{f'}"', bend right=40, from=1-1, to=1-2]
        \arrow[""{name=2, anchor=center, inner sep=0}, "g", bend left=40, from=1-2, to=1-3]
        \arrow[""{name=3, anchor=center, inner sep=0}, "{g'}"', bend right=40, from=1-2, to=1-3]
        \arrow["\alpha", shorten <=3pt, shorten >=3pt, Rightarrow, from=0, to=1]
        \arrow["\beta", shorten <=3pt, shorten >=3pt, Rightarrow, from=2, to=3]
    \end{tikzcd}
\end{center}
As a ps-context it is given by
\begin{align*}
    O^2_0(\alpha,\beta) :\equiv x : \Ob, \, &y : \Ob, f : x \to y, f' : x \to y, \alpha : f \to f', \\
    &z : \Ob, g : y \to z, g' : y \to z, \beta : g \to g'
\end{align*}
in which we may construct the horizontal composite $ \alpha \mast{2}{0} \beta : f \mast{1}{0} g \to f' \mast{1}{0} g' $, where $ \mast{1}{0} $ denotes the composition of 1-cells.

We now construct all binary composites more systematically. The $ d$-dimensional globe as a context can be defined using the rules for ps-context. First of all, by (PSS) we have $ x : \Ob \vdashps x : \Ob $ and we define $ D^0(x : \Ob) :\equiv x : \Ob $. By construction, the judgment $ D^0(x:\Ob) \vdashps x : \Ob $ is derivable and $ \dim(x) = 0 $. Applying inductively the rule (PSE) we get
\begin{center}
    \begin{prooftree}
        \hypo{D^d(x : A) \vdashps x : A}
        \infer1{D^d(x : A), y : A, f : x \to_A y \vdashps f : x \to_A y}
    \end{prooftree}
\end{center}
and we define $ D^{d+1}(f : x \to y) :\equiv D^d(x : A), y : A, f : x \to_A y $. Given a $d$-dimensional globe $ D^d(x:A)$, by an inductive argument one can show that $ \dim(x) = d $ as well as $ \FV(D^d(x : A)) = \FV(x : A) $. Often we will suppress the type and simply write $ D^d(x) $ instead of $ D^d(x : A) $.

Next we define a collection of ps-contexts, the gluings of which will give us the binary compositions. Let $ n \in \bN $ be some natural number. Applying (PSD) and then (PSE) again to $ D^{n+1}(f : x \to_A y) \vdashps f : x \to_A y $ we get a context
\begin{align*}
    \overbrace{D^{n+1}(f : x \to_A y), \, z : A, \, g : y \to z }^{ O^{n+1}_{n}(f \, : \, x \to_A y, \, g \, : \, y \to_A z) \ :\equiv} \vdashps g : y \to_A z.
\end{align*}
Again for brevity we may also write $ O^{n+1}_n(f,g) $. These contexts will allow us to define sequential composition (or vertical composition in the 2-dimensional case). As a sequence of rules this judgment is represented by $ \textup{(PSS)(PSE)}^{d+1}\textup{(PSD)(PSE)}$.

For the contexts of the remaining compositions we take the ps-context which as a sequence of rules is given by $ \textup{(PSS)}\textup{(PSE)}^d \textup{(PSD)}^{d-n}\textup{(PSE)}^{d-n} $. The judgment obtained is of the form
\begin{align*}
    \overbrace{D^d(\alpha : X), \Delta, \beta : Y}^{O^{d}_{n}(\alpha \, : \, X, \, \beta \, : \, Y) \ \equiv} \vdashps \beta : Y
\end{align*}
where $ d > n $, $ \Delta $ is a string of variables containing the target of $ \beta $ and variables thereof.

Now, a calculation shows that
\begin{align} \label{eq:On1nBoundary}
    \begin{split}
        \partial^- O^{n+1}_n(f : x \to_A y, g : y \to_A z) &\equiv D^{n}(x : A) \\[2mm] 
        \partial^+ O^{n+1}_n(f : x \to_A y, g : y \to_A z) &\equiv D^{n}(z : A)    
    \end{split}
\end{align}
as well as
\begin{align} \label{eq:Od1nBoundary}
    \begin{split}
        \partial^- O^{d+1}_n(\phi : \alpha \to_X \alpha', \psi : \beta \to_Y \beta') &\equiv O^d_n(\alpha : X, \beta: Y) \\[2mm]
        \partial^+ O^{d+1}_n(\phi : \alpha \to_X \alpha', \psi : \beta \to_Y \beta') &\equiv O^d_n(\alpha' : X, \beta': Y).
    \end{split}
\end{align}

We use these ps-context to construct all binary operations $ \mast{d}{n} $. By equation~\ref{eq:On1nBoundary} and with the abbreviation $ O^{n+1}_n(f,g) $ for $ O^{n+1}_n(f : x \to_A y, \, g : y \to_A z) $ we have
\begin{center}
    \begin{prooftree}
        \hypo{ O^{n+1}_n(f,g) \vdashps}
        \hypo{D^n(x : A) \vdash x : A}
        \hypo{D^n(z : A) \vdash z : A}
        \hypo{ \Delta \vdash \gamma :  O^{n+1}_n(f,g) }
        \infer4[\sOP]{ \Delta \vdash \gamma_f \mast{n+1}{n} \gamma_g : \gamma_x \to_{A[\gamma]} \gamma_z }
    \end{prooftree}
\end{center}
where we have defined $ \gamma_f \mast{n+1}{n} \gamma_g :\equiv \op_{ O^{n+1}_{n}(f,g)}[\gamma] $.

The remaining binary compositions are built by induction. To start consider the ps-context $ O^{d+1}_n(\phi : \alpha \to \alpha', \psi : \beta \to \beta') $. By the inductive hypothesis, the terms $ O^d_n(\alpha, \beta) \vdash \alpha \mast{d}{n} \beta : T $ and $ O^d_n(\alpha', \beta') \vdash \alpha' \mast{d}{n} \beta' : T $ are derivable and parallel. This allows us to apply the (OP) rule in the following way:
\begin{center}
	\adjustbox{scale=0.9}{
    \begin{prooftree}
        \hypo{O^{d+1}_n(\phi,\psi) \vdashps}
        \hypo{O^d_n(\alpha, \beta) \vdash \alpha \mast{d}{n} \beta : T}
        \hypo{O^d_n(\alpha', \beta') \vdash \alpha' \mast{d}{n} \beta' : T}
        \hypo{ \Delta \vdash \gamma : O^{d+1}_n(\phi,\psi)}
        \infer4[\sOP]{ \Delta \vdash \phi \mast{d+1}{n} \psi : \alpha \mast{d}{n} \beta \to \alpha' \mast{d}{n} \beta' }
    \end{prooftree}
	}
\end{center}
where we have used equation~\ref{eq:Od1nBoundary} and where we defined $ \phi \mast{d+1}{n} \psi \equiv \op_{O^{d+1}_n(\phi,\psi)}[\gamma] $.

A data of a context morphisms $ \Delta \vdash \gamma : O^d_n(\phi,\psi) $ is equivalently given by two $ d$-dimensional terms $ \Delta \vdash u : s \to t $ and $ \Delta \vdash v : s' \to t' $ such that $ \cod^{d-n}(u) \equiv \dom^{d-n}(v) $. Thus, for any two such terms we may derive their composite $u \mast{d}{n} v $, which will also use the notation $ u \cdot v $ when $ d = n+1 $. Given a substitution $ \Theta \vdash \delta : \Delta $ we have $ (u \mast{d}{n} v)[\delta] \equiv u[\delta] \mast{d}{n} v[\delta] $.

Consider the globe context $ D^d(x : A) $. Given any term $ \Delta \vdash t : T $ let us set $ 1^0_t :\equiv t $. Applying the (COH) rule inductively then produces all higher identities
\begin{center}
    \begin{prooftree}
        \hypo{D^d(x : A) \vdashps}
        \hypo{D^d(x : A) \vdash 1^n_x \to 1^n_x}
        \hypo{\Delta \vdash \gamma : D^d(x : A)}
        \infer3[\sCOH]{\Delta \vdash 1^{n+1}_{\gamma_x} : 1^n_{\gamma_x} \to 1^n_{\gamma_x}}
    \end{prooftree}
\end{center}
where we have defined $ 1^{n+1}_{\gamma_x} :\equiv \coh_{D^d(x : A)}[\gamma] $. In the case of $ n = 1 $ we also write $ 1_{\gamma_x} $ for $ 1^1_{\gamma_x} $. Since a term $ \Delta \vdash t : T $ contains the same information as a context morphism $ \Delta \vdash \hat t : D^d(x : A) $ with $ \hat t_x \equiv t $ (see Benjamin--Finster--Mimram~\cite{benjamin2021globular} Lemma~74), we can construct the identity $ 1^{n+1}_t : 1^n_t \to 1^n_t $ for any $d$-dimensional term $ \Delta \vdash t : T $. Given a substitution $ \Theta \vdash \delta : \Delta $ we have $ 1^n_t[\delta] \equiv 1^n_{t[\delta]}$.

%%%%%%%%%%%%%%%%%%%%%%%%%%%%%%%%%%%%%%%%%%%%%%%%%%%%%%
%%%%%%%%%%%%%%%%%%%%%%%%%%%%%%%%%%%%%%%%%%%%%%%%%%%%%%
\section{The universal cone} \label{sec:UniversalCone}
%%%%%%%%%%%%%%%%%%%%%%%%%%%%%%%%%%%%%%%%%%%%%%%%%%%%%%
%%%%%%%%%%%%%%%%%%%%%%%%%%%%%%%%%%%%%%%%%%%%%%%%%%%%%%

In 1-category theory, thanks to the free-forgetful adjunction between directed graphs and 1-categories, limits of functors can be equivalently described as limits over diagrams. When considering $(\infty,\infty)$-categories, the appropriate notion encapsulating the generating data of free structures is that of computads. By a result of Benjamin, Markakis and Sarti~\cite{Benjamin_2024}, the contexts of \texttt{CaTT} are precisely finite computads. We exploit this fact to build a theory of limits over finite computads in \texttt{CaTT} in which the role of the diagram is played by the contexts. Note that the restriction to finite computads is a limitations of type theory, in which contexts are necessarily finite. In principle, the same ideas could be used to define arbitrary limits in a framework without this limitation (e.g. in that of computads as in the work of Dean, Finster, Markakis, Reutter and Vicary~\cite{dean2024computads}).

The constituent cells of the cones and the corresponding limit notions we consider here are oriented. As a result, when building cones there are number of choices we can make with respect to the orientation of the terms involved, namely one for each dimension. As an example, in the 2-dimensional case we have four options. The first choice fixes the orientation of the 1-cells leading to the notion of limits and colimits respectively. The second choice fixes the orientation of the 2-cells, which are referred to as the lax and oplax versions of the corresponding (co)limit notion. In our case, for the purposes of this section we will be working with a choice of orientation with which the construction acquires a convenient uniformity. All other choices can be obtained by changing the roles of the terms $ s $ and $ t $ in the definition appropriately. Alternatively, all other choices can be obtained by taking opposites, which are studied by Benjamin and Markakis~\cite{benjamin2024duality} for $(\infty,\infty)$-categories in the related framework of computads of Dean et al~\cite{dean2024computads}.

To motivate the following proposition, let us begin with an example and build a cone over a 2-globe. The context of a 2-globe is given by    
\begin{align*}
    \Gamma \equiv x : \Ob, y : \Ob, f : x \to y, g : x \to y, \alpha : f \to g \qquad\qquad
\end{align*}
and a cone over this diagram is given by
\begin{center}
    \begin{tikzcd}[ampersand replacement=\&]
        {c} \\
        \& x \\
        \& y
        \arrow["{p_x}", from=1-1, to=2-2]
        \arrow[""{name=0, anchor=center, inner sep=0}, "{p_y}"', bend right=40, from=1-1, to=3-2]
        \arrow["g", bend left=30, from=2-2, to=3-2]
        \arrow["{p_g}", shorten <=4pt, Rightarrow, from=0, to=2-2]
    \end{tikzcd}
    \qquad $ \overset{p_\alpha}{\equiv\mkern-2.5mu\equiv\mkern-2.5mu\Rrightarrow} $ \qquad
    \begin{tikzcd}[ampersand replacement=\&]
        {c} \\
        \& x \\
        \& y
        \arrow["{p_x}", from=1-1, to=2-2]
        \arrow[""{name=0, anchor=center, inner sep=0}, "{p_y}"', bend right=40, from=1-1, to=3-2]
        \arrow[""{name=1, anchor=center, inner sep=0}, "f"', bend right=30, from=2-2, to=3-2]
        \arrow[""{name=2, anchor=center, inner sep=0}, "g", bend left=30, from=2-2, to=3-2]
        \arrow["{p_f}", shorten <=4pt, Rightarrow, from=0, to=2-2]
        \arrow["\alpha", shorten <=2pt, shorten >=2pt, Rightarrow, from=1, to=2]
    \end{tikzcd}
\end{center}
which as a context has the form
\begin{align*}
    C(\Gamma) \equiv \ &x : \Ob, \ y : \Ob, \ f : x \to y, \ g : x \to y, \alpha : f \to g \\
    &c : \Ob, \ p_x : c \to x, \ p_y : c \to y, \\\
    &p_f : p_y \to p_x \mast{1}{0} f, \ p_g : p_y \to p_x \mast{1}{0} g, \\
    &p_\alpha : p_g \to p_f \mast{2}{1} \left( 1_{p_x} \mast{2}{0} \alpha \right).
\end{align*}

Since we are working with a weak category, the construction of cones involves choices, as there are many ways of gluing a given diagram. The choices we made are for the sake of convenience because they fit into a pattern revealing certain features we will eventually rely on in our definition.

We can already see these features in the cone as a context in this simple example. The cone is generated by ``formally adding'' an apex and then freely adding a unique ``projection'' to each variable in $ \Gamma $. Moreover, the types of each projection all have common features, which become more apparent if we increases the dimension and the complexity of the diagrams.

Take, for example, $ p_\alpha : p_g \to p_f \mast{2}{1} \left( 1_{p_x} \mast{2}{0} \alpha \right) $. We see, first of all, that the source of $ p_\alpha $ involves only the projections associated to the target of $ \alpha $, namely $ g $. Note that the projections associated to the dependencies of $ g $, namely $ p_x $ and $ p_y $ are themselves dependencies of $ p_g $ and thus appear in the construction of $ p_g $ and its type. On the other hand, the target of $ p_\alpha $ involves only projections associated to the source of $ \alpha $, namely $ f $. The target of $ p_\alpha $ also involves the variable $ \alpha $ itself. Moreover, $ \alpha $ appears in a certain linear way which we now make this precise.

\begin{definition} \label{def:LinearVar}
    Let $ \Gamma \vdash t : A $ be a term and let $ x \in \FV(t) $ be a variable such that $ \dim(x) = \dim(t) $. We define the relation $ x \propto t $ by inducting over the term $ t $ as follows.
    \begin{enumerate}
        \item[\textup{(VAR)}] $ \Gamma \vdash y : A $ where $ (y : A) \in \Gamma $
        
        $ x \propto y $ if $ x \equiv y $.

        \item[\textup{(OP)}] $ \Delta \vdash \textup{\texttt{op}}_{\Gamma,s \to_At}[\gamma] : s[\gamma] \to_{A[\gamma]} t[\gamma] $
        
        $ x \propto \textup{\texttt{\op}}_{\Gamma,s \to_A t}[\gamma] $ if there exists a unique $ x_i \in \FV(\Gamma) $ such that $ x \propto \gamma_i $.

        \item[\textup{(COH)}] $ \Delta \vdash \textup{\texttt{\coh}}_{\Gamma, s\to_A t}[\gamma] : s[\gamma] \to_{A[\gamma]} t[\gamma] $
        
        This case is the same as that for \textup{(OP)}.
    \end{enumerate}
    If $ x \propto t $ we say $ x $ is linear in $ t $.
\end{definition}

%%%%%%%%%%%%%%%%%%%%%%%%%%%%%%%%%%%%%%%%%%%%%%%%%%%%%%
%%%%%%%%%%%%%%%%%%%%%%%%%%%%%%%%%%%%%%%%%%%%%%%%%%%%%%
\subsection{Cones as context} \label{subsec:ConesAsContext}
%%%%%%%%%%%%%%%%%%%%%%%%%%%%%%%%%%%%%%%%%%%%%%%%%%%%%%
%%%%%%%%%%%%%%%%%%%%%%%%%%%%%%%%%%%%%%%%%%%%%%%%%%%%%%

Motivated by the previous section we can now spell out a set of rules which build cones over a diagram. Notably we do not restrict ourselves to any particular way of gluing. All legitimate ways of gluing cells to obtain higher projections are permitted and lead to distinct but valid cones. 

The idea is to induct over the underlying context, adding precisely one higher projection for each such new variable. For the sake of readability and reusability, let us introduce a shorthand for the side conditions that will appear in the rules. Given a term $ K \vdash t : A $, a variable $ (x : X) \in K $ and a sublist of variables $ \Pi $ of $K$, we let $ \delta\text{Cond}(t : A, x : X, c : \Ob, \Pi) $ denote the three conditions:
\begin{itemize}
    \item[--] $ t : A $ is categorical
    \item[--] $ \FV(t : A) = \FV(x : X) \cup \bigcup\limits_{ \substack{p \, \in \, \FV(\Pi) \\[1mm] y \, \in \, \FV(\delta(x) \, : \, \partial X) \\[1mm] y \, \propto \, \tau(p) } } \FV(p : T_p) $
    \item[--] $ x \propto t $
\end{itemize}
where $ \delta $ stands for $ \sigma $ or $ \tau $. We will also denote by $ \delta\overline{\text{Cond}}(t : A, x : X, \Pi) $ the same set of conditions except for the linearity condition, and with $\FV(x:X)$ in $\FV(t:A)$ replaced by $ \{c\} $. In the following rules, the set $ \Pi $ will be such that it contains precisely the projections. These conditions ensure that the term $ t $ is built using the appropriate projections. 

\paragraph{Rules for cones as contexts.}
We introduce a new auxiliary judgment, denoted by $K \, \texttt{cone} \, (\Gamma;c) $. This judgement is subject to the rules
    \bigskip
    \begin{center}
        \begin{prooftree}
            \infer0[\textup{\footnotesize{(EK)}}]{ c : \Ob \ \cone \ (\emptyset; c)}
        \end{prooftree}
    \end{center}
    \medskip
    \begin{center}
        \adjustbox{scale=0.95}{
        \begin{prooftree}
            \hypo{\Gamma, c : \Ob, \Pi \ \cone \ (\Gamma;c)}
            \hypo{\Gamma, x : X, c : \Ob, \Pi \vdash s \to_A t}
            \infer2[\textup{\footnotesize{(KE)}}]{\Gamma, x : X, c : \Ob, \Pi, p_x : s \to_A t \ \cone \ ((\Gamma,x : X);c)}
        \end{prooftree}
        \quad \footnotesize
        \begin{tabular}{l}
            $ \tau\overline{\textup{Cond}}(s : A, x : X, c : \Ob, \Pi) $ \\[1mm]
            $\sigma\textup{Cond}(t : A, x : X, \Pi) $
        \end{tabular} \normalsize
        }
    \end{center}
    \bigskip
    If the judgment $ K \,\textup{\texttt{cone}}\, (\Gamma;c) $ is derivable we say $ K $ is a cone over $\Gamma $ with apex $ c $. The variable $ p_x $ is called the projection corresponding to the variable $ x $.

The first rule, to be thought of as a base case, says: a cone over the empty diagram is just the apex. The second rule, the inductive step, assumes we are given a cone $ K $ over a diagram $ \Gamma $. Then, given a context extension $ \Gamma, x : X $ we first of all ask whether a certain type $ s \to_A t $ can be built in $ K $ extended by $ x : X $. This type together with the variable $ x \in \FV(\Gamma)$ must satisfy certain conditions, ensuring that it is of the form as in the previous section (up to coherence). If this is the case, then we append a new variable $ p_x $ to $ K $, which we think of as the (higher) projection corresponding to the variable $ x \in \FV(\Gamma)$.

The following lemma assures us that the above rules are well formed, allowing us to interpret them as a recognition algorithm for context which are cones over a given context. Recall, here that the length of a context $\Gamma$ is denote by $ |\Gamma|$.

\begin{lemma} \label{lem:ConeAdmiss}
    The rule
    \begin{center}
        \begin{prooftree}
            \hypo{K \ \cone \ (\Gamma;c)}
            \infer1{K \vdash}
        \end{prooftree}
    \end{center}
    \bigskip
    is admissible. Moreover, given a judgment $ K \ \cone \ (\Gamma;c) $, the context $ K $ is of the form $ \Gamma, c : \Ob, \Pi $ for some list of variables $ \Pi $ such that $ |\Gamma|=|\Pi|$.
\end{lemma}

\begin{proof}
    By induction.
\end{proof}

%%%%%%%%%%%%%%%%%%%%%%%%%%%%%%%%%%%%%%%%%%%%%%%%%%%%%%%
%%%%%%%%%%%%%%%%%%%%%%%%%%%%%%%%%%%%%%%%%%%%%%%%%%%%%%%
\subsection{Existence of cones} \label{subsec:ConesGlobCont}
%%%%%%%%%%%%%%%%%%%%%%%%%%%%%%%%%%%%%%%%%%%%%%%%%%%%%%%
%%%%%%%%%%%%%%%%%%%%%%%%%%%%%%%%%%%%%%%%%%%%%%%%%%%%%%%

In the situation where the underlying diagram $ \Gamma $ is globular we can give an explicit construction and provide an existence proof for the cone.

\begin{theorem} \label{prop:ConePsContext}
    Let $ \Gamma \vdash $ be a globular context. There exists a context
    \begin{align*}
        C(\Gamma) \equiv \Gamma, \ c : \Ob, \ \Pi, \qquad \text{where}\qquad \Gamma \equiv (x_i : A_i)_{1 \le i \le n} \text{ and } \Pi \equiv (p_{x_i} : T_{x_i})_{1 \le i \le n}
    \end{align*}
    such that $ C(\Gamma) \, \textup{\texttt{cone}} \, (\Gamma, c) $ is derivable. Moreover, the type $ T_x $ of the projection $ p_{x} $ can be taken to be of the form
    \begin{align*}
        &c \to x, \qquad &&\text{if } d = 0 \\
        &p_{\tau(x)} \to p_{\sigma(x)} \mast{d}{d-1} \Big(1_{p_{\sigma^2(x)}}^d \mast{d}{d-2} \cdots \Big(1_{p_{\sigma^{d}(x)}}^d \mast{d}{0} x \Big) \Big), &&\text{if } d > 0.
    \end{align*}
    where $ d = \dim(x) $.
\end{theorem}

\begin{proof}
    We prove the statement by induction over the length of $ \Gamma $. As part of the inductive process we also show that for a given variables $ (x : A) \in \Gamma $, the free variables of the corresponding projections are given by
    \begin{align} \label{eq:FVproj}
        \FV(p_x : T_x) = \bigcup_{ \substack{p_{x_i} \in \FV(\Pi) \\ y \in \FV(x \, : \, A) \\ y \, \propto \, \tau(p_{x_i}) }} \FV(p_{x_i} : T_{x_i}).
    \end{align}
    In addition to that we have $ c \in \FV(p_x : T_x) $ and $ x $ is the unique variable in $ \Gamma $ such that $ x \propto \tau(p_x) $.
    
    If $ \Gamma \equiv \emptyset $, then the rules directly give us $ c : \Ob $ as a cone over $ \emptyset $ with apex $ c $. Assume now that we have built the cone for the context $ \Gamma $, that is, we have a derivation of $ C(\Gamma) \texttt{ cone } (\Gamma;c) $ as described in the statement. Assume moreover that we now extend the context $\Gamma $ to $ \Gamma, x : A $. If $ \dim(x) = 0$, then $ c \to x $ is derivable and satisfies the condition of the rule (KE). Thus, applying (KE) we can extend the context by $ p_x : c \to x $, giving us a cone over $ \Gamma, x : A $. Moreover, $ p_x : c \to x $ satisfies equation~\ref{eq:FVproj}.
    
    Assume now that $ d := \dim(x) > 0 $ and define the terms
    \begin{align}
        t^{d+1}(x) :\equiv x : A, \qquad\qquad t^{n}(x) :\equiv 1^{n-1}_{p_{\sigma^{n}(x)}} \mast{d}{d-n} t^{n+1}(x) : T_n(x)
    \end{align}
    for $ 1 \le n \le d $ and where $ T $ is the appropriate type. The term $ t^n(x) $, which is of the form
    \begin{align*}
        t^{n}(x) :\equiv 1^{n-1}_{p_{\sigma^{n}(x)}} \mast{d}{d-n} \Big( 1^{n}_{p_{\sigma^{n+1}(x)}} \mast{d}{d-n-1} \cdots \Big( 1^{d-1}_{p_{\sigma^d(x)}} \mast{d}{0} x \Big) \Big).
    \end{align*}
    approximates $ \tau(p_x) $ as $ n $ decreases. We claim that for all $ 1 \le n \le d+1 $ we have
    \begin{enumerate}[label=(\roman*)]
        \item \label{enu:item1} $ t^{n}(x) : T_n(x) $ derivable
        \item \label{enu:item2} $ \sigma^m(t^n(x)) \equiv t^{n-m}(\sigma^m(x)) $ for all $ 0 \le m < n $
        \item \label{enu:item3} $ \tau(t^n(x)) \equiv t^{n-1}(\tau(x)) $ if $ 2 \le n $
        \item \label{enu:item4} $ \FV(t^n(x) : T_n(x)) = \FV(x : A) \cup \FV(p_{\sigma^n(x)} : T_{\sigma^n(x)}) $ if $ n < d + 1$
        \item \label{enu:item5} $ x \propto t^n(x) $
    \end{enumerate}

    For the statements all the statements we now perform a nested induction on $ n $ starting with $ n = d+1 $ all the way down to $ n = 1 $. We start with~\ref{enu:item1} and~\ref{enu:item2}, which we prove simultaneously. For $ n = d+1 $ we have $ t^{d+1}(x) \equiv x : A $ which is derivable so that~\ref{enu:item1} holds. Now since $ \sigma^m(x) $ is of dimension $ \dim(x) - m $ we also have $ t^{d+1-m}(\sigma^m(x)) \equiv \sigma^m(x) $ which proves~\ref{enu:item2}. Assume now that the statements hold for some $ 1 < n $. For~\ref{enu:item1} we compute
    \begin{align} \label{eq:ConePropSourceIsTarget}
        \sigma^{n-1}(t^n(x)) \equiv t^1(\sigma^{n-1}(x)) \equiv \tau(p_{\sigma^{n-1}(x)}) \equiv \tau^{n-1}\left(1^{n-2}_{p_{\sigma^{n-1}(x)}}\right).
    \end{align}
    Here the first and second identities hold by the inductive hypothesis. Notice that in order to define $ t^{n-1}(x) $ starting with $ t^n(x) $ we also made use of $ p_{\sigma^{n-1}(x)} $. But $ \sigma^{n-1}(x) $ is a variable which necessarily appears before $ x $, so that the context currently at hand already contains the corresponds projection, allowing us to form this term. From equation~\ref{eq:ConePropSourceIsTarget} we conclude that the two terms in
    \begin{align*}
        t^{n-1}(x) \equiv 1^{n-2}_{p_{\sigma^{n-1}(x)}} \mast{d}{d-n+1} t^n(x) : T_{n-1}(x)
    \end{align*}
    are composable, rendering the composite derivable. Regarding~\ref{enu:item2}, for all $ 0 \le m < n-1 $ we have
    \begin{align*}
        \sigma^m(t^{n-1}(x)) &\equiv \sigma^m\left(1^{n-2}_{p_{\sigma^{n-1}(x)}} \mast{d}{d-n+1} t^n(x)\right) \\
        &\equiv \sigma^m\left(1^{n-2}_{p_{\sigma^{n-1}(x)}} \right) \mast{d-m}{d-n+1} \sigma^m (t^n(x)) \\
        &\equiv 1^{n-m-2}_{p_{\sigma^{n-1}(x)}} \mast{d-m}{d-n+1} t^{n-m}(\sigma^m(x)) \\
        &\equiv t^{n-m-1}(\sigma^m(x))
    \end{align*}
    where in the second line we used the fact that~\ref{enu:item2} holds for $ t^n(x) $.
    
    Now, for~\ref{enu:item3}, an argument similar to that for the base case of~\ref{enu:item2} shows that~\ref{enu:item3} holds which applies since $ n = d + 1 \ge 2 $. Now, assuming $ n-1 \ge 2 $, we compute
    \begin{align*}
        \tau(t^{n-1}(x)) &\equiv \tau\left(1^{n-2}_{p_{\sigma^{n-1}(x)}} \mast{d}{d-n+1} t^n(x) \right) \\
        &\equiv \tau\left( 1^{n-2}_{p_{\sigma^{n-1}(x)}} \right) \mast{d-1}{d-n+1} \tau(t^n(x)) \\
        &\equiv 1^{n-3}_{p_{\sigma^{n-1}(x)}} \mast{d-1}{d-n-1} t^{n-1}(\tau(x)) \\
        &\equiv 1^{n-3}_{p_{\sigma^{n-2}(\tau(x))}} \mast{d-1}{d-n-1} t^{n-1}(\tau(x)) \\
        &\equiv t^{n-2}(\tau(x)).
    \end{align*}

    For~\ref{enu:item4} we have $ \FV(t^d(x)) = \FV( 1^{d-1}_{p_{\sigma^d(x)}} \mast{d}{0} x) = \FV(p_{\sigma^d(x)} : T_{\sigma^d(x)}) \cup \FV(x : A) $. For the inductive step,
    \begin{align*}
        \FV(t^{n}(x) : T_n(x)) &= \FV\left(1^{n-1}_{p_{\sigma^{n}(x)}} \mast{d}{d-n} t^{n+1}(x) : T_n(x) \right) \\
        &= \FV(p_{\sigma^n(x)} : T_{\sigma^n(x)}) \cup \FV(t^{n+1}(x) : T_n(x)) \\
        &= \FV(p_{\sigma^n(x)} : T_{\sigma^n(x)}) \cup \FV(p_{\sigma^{n+1}(x)} : T_{\sigma^{n+1}(x)}) \cup \FV(x : A) \\
        &= \FV(p_{\sigma^n(x)} : T_{\sigma^n(x)}) \cup \FV(x : A)
    \end{align*}
    where we used the fact that $ \FV(p_{\sigma^{n+1}(x)} : T_{\sigma^{n+1}(x)}) \subset \FV(p_{\sigma^n(x)} : T_{\sigma^n(x)}) $.

    As for~\ref{enu:item5}, by definition $ t^{d+1}(x) \equiv x : A $, and indeed, $ x $ is linear in $ x $ by definition. Now assume that $ x \propto t^n(x) $ for some $ 1 \le n $ and we ask if $ x $ is linear in $ t^{n+1}(x) \equiv 1^{n-1}_{p_{\sigma^n}(x)} \mast{d}{d-n} t^{n+1}(x) : T_{n-1}(x) $, which is the case, as it only appears in $ t^{n+1}(x) $.

    Using the properties~\ref{enu:item2} and~\ref{enu:item3}, we find
    \begin{gather*}
        \tau(t^1(x)) \equiv \tau \left(p_{\sigma(x)} \mast{d}{d-1} t^2 \right) \equiv \tau(t^2(x)) \equiv t^1(\tau(x)) \equiv \tau(p_{\tau(x)}) \\[2mm]
        \sigma(t^1(x)) \equiv \sigma\left(p_{\sigma(x)} \mast{d}{d-1} t^2 \right) \equiv \sigma\left(p_{\sigma(x)}\right) \equiv p_{\sigma^2(x)} \equiv p_{\sigma \tau(x)} \equiv \sigma\left(p_{\tau(x)}\right)
    \end{gather*}
    from which we conclude $ t^1(x) \parallel p_{\tau(x)} $. This allows us to derive the judgment $ p_{\tau(x)} \to t^1(x) $.

    Now, by the inductive hypothesis
    \begin{align*}
        \FV(p_{\tau(x)} : T_{\tau(x)}) = \FV(c: \Ob) \cup \bigcup_{ \substack{p_{x_i} \in \FV(\Pi) \\ y \in \FV(\tau(x) \, : \, T_{\tau(x)}) \\ y \, \propto \, \tau(p_{x_i}) }} \FV(p_{x_i} : T_{x_i}).
    \end{align*}
    On the other hand, using~\ref{enu:item4} we have
    \begin{align*}
        \FV(t^1(x) : T_1(x)) &= \FV(x : A) \cup \FV(p_{\sigma(x)} : T_{\sigma(x)}) \\
        &= \FV(x : A) \cup \bigcup_{ \substack{p_{x_i} \in \FV(\Pi) \\ y \in \FV(\sigma(x) \, : \, T_{\sigma(x)}) \\ y \, \propto \, \tau(p_{x_i}) }} \FV(p_{x_i} : T_{x_i})
    \end{align*}
    which shows that $ p_{\tau(x)} \to t^1(x) $ satisfies the conditions of (KE) allowing us to produce the cone $ C(\Gamma,x:A) \equiv \Gamma, x : A, c : \Ob, \Pi, p_x : p_{\tau(x)} \to t^1(x) $.

    To complete the induction we need to show that equation~\ref{eq:FVproj} still holds for all variables in $ \Gamma $. On top of that we need to show that $ x : A $ satisfies equation~\ref{eq:FVproj} as well as the two properties stated right after. We start with the latter.

    A quick inductive argument shows that $ c \in \FV(p_x : T_x) $. Also, by~\ref{enu:item4}
    \begin{align*}
        \FV(p_x : T_x) &= \FV(p_{\tau(x)} : T_{\tau(x)}) \cup \FV(t^1(x) : T_1(x)) \\
        &= \FV(p_{\tau(x)} : T_{\tau(x)}) \cup \FV(p_{\sigma(x)} : T_{\sigma(x)}) \cup \FV(x : A).
    \end{align*}
    Since $ \sigma(x) \propto \tau(p_{\sigma(x)}) $ and $ \tau(x) \propto \tau(p_{\tau(x)}) $, we know that the variables $ p_{\sigma(x)}, p_{\tau(x)}$ and $x$ all have the same dimension. Moreover, by Lemma~\ref{lem:FreeVariablesBoundTypeDim} these are the only variables of dimension $ \dim(x) $. Thus $ x $ is the unique variable of $ \Gamma $ such that $ x \propto \tau(p_x) $. This fact guarantees that $ p_x $ satisfies equation~\ref{eq:FVproj} and that all the projections in $ \Pi $ continue to satisfy the same equation in the extended cone over $ \Gamma, x : A $.
\end{proof}

It is possible to give algorithmic constructions for the cones also for more complicated diagrams in the strict world. In the following examples we will suppress the labels of the binary operations in the interest of making the equations more readable and the global form more evident.

\begin{example}
    Consider the diagram
    
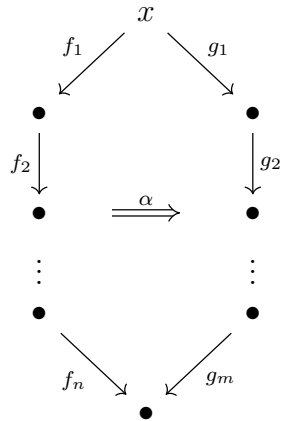
\begin{center}
        \begin{tikzcd}[row sep={2em}, column sep={2em}]
            & x \\
            \bullet && \bullet \\
            \bullet && \bullet \\
            \bullet && \bullet \\
            & \bullet
            \arrow["{f_1}"', from=1-2, to=2-1]
            \arrow["{g_1}", from=1-2, to=2-3]
            \arrow["{f_2}"', from=2-1, to=3-1]
            \arrow["{g_2}", from=2-3, to=3-3]
            \arrow["\alpha", shorten <=19pt, shorten >=19pt, Rightarrow, from=3-1, to=3-3]
            \arrow["\vdots"{description}, draw=none, from=3-1, to=4-1]
            \arrow["\vdots"{description}, draw=none, from=3-3, to=4-3]
            \arrow["{f_n}"', from=4-1, to=5-2]
            \arrow["{g_m}", from=4-3, to=5-2]
        \end{tikzcd}
    \end{center}
    Then, in the strict case, the projection corresponding to the variable $ \alpha $ can be taken to be of the type:
    \begin{gather*}
        p_{g_m} \cdot (p_{g_{m-1}} \ast 1_{g_m}) \cdot (p_{g_{m-2}} \ast 1_{g_{m-1}\cdot g_m}) \cdots (p_{g_1} \ast 1_{g_2 \cdots g_m}) \\
        \downarrow \\
        p_{f_n} \cdot (p_{f_{n-1}} \ast 1_{f_n}) \cdot (p_{f_{n-2}} \ast 1_{f_{n-1}\cdot f_n}) \cdots (p_{f_1} \ast 1_{f_2 \cdots f_n}) \cdot (1_{p_x} \ast \alpha).
    \end{gather*}
    By defining the terms $ p_{f_1 \cdots f_n} $ and  $ p_{g_1 \cdots g_m} $ appropriately we can rewrite this as
    \begin{align*}
        p_{g_1 \cdots g_m} \to p_{f_1 \cdots f_n} \cdot (1_{p_x} \ast \alpha).
    \end{align*}
\end{example}

\begin{example}
    Consider the diagram
    
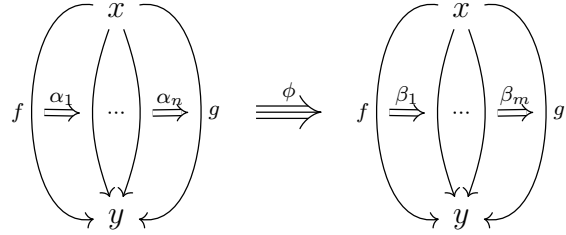
\begin{center}
        \begin{tikzcd}[row sep=large]
            x &&&& x \\
            \\
            y &&&& y
            \arrow[""{name=0, anchor=center, inner sep=0}, "f"', bend right=100, from=1-1, to=3-1]
            \arrow[""{name=1, anchor=center, inner sep=0}, "g", bend left=100, from=1-1, to=3-1]
            \arrow[""{name=2, anchor=center, inner sep=0}, bend right=20, from=1-1, to=3-1]
            \arrow[""{name=3, anchor=center, inner sep=0}, bend left=20, from=1-1, to=3-1]
            \arrow[""{name=4, anchor=center, inner sep=0}, "f"', bend right=100, from=1-5, to=3-5]
            \arrow[""{name=5, anchor=center, inner sep=0}, "g", bend left=100, from=1-5, to=3-5]
            \arrow[""{name=6, anchor=center, inner sep=0}, bend right=20, from=1-5, to=3-5]
            \arrow[""{name=7, anchor=center, inner sep=0}, bend left=20, from=1-5, to=3-5]
            \arrow["{\alpha_1}", shorten <=4pt, shorten >=4pt, Rightarrow, from=0, to=2]
            \arrow["\cdots"{description}, draw=none, from=2, to=3]
            \arrow["{\alpha_n}", shorten <=4pt, shorten >=4pt, Rightarrow, from=3, to=1]
            \arrow["\phi", shorten <=20pt, shorten >=20pt, Rightarrow, scaling nfold=3, from=1, to=4]
            \arrow["{\beta_1}", shorten <=4pt, shorten >=4pt, Rightarrow, from=4, to=6]
            \arrow["\cdots"{description}, draw=none, from=6, to=7]
            \arrow["{\beta_m}", shorten <=4pt, shorten >=4pt, Rightarrow, from=7, to=5]
        \end{tikzcd}
    \end{center}
    For this diagram, again in the strict case, the projection corresponding to the variable $ \phi $ can be taken to be of the type:
    % \scalebox{0.9}{
    \begin{gather*}
        p_{\beta_m} \cdot (p_{\beta_{m-1}} \ast 1_{1_{p_x} \ast \beta_m}) \cdot (p_{\beta_{m-2}} \ast 1_{1_{p_x} \ast (\beta_{m-1}\cdot \beta_m)}) \cdots (p_{\beta_1} \ast 1_{1_{p_x} \ast (\beta_2 \cdots \beta_m)}) \\
        \downarrow \\
        p_{\alpha_n} \cdot (p_{\alpha_{n-1}} \ast 1_{1_{p_x} \ast \alpha_n}) \cdot (p_{\alpha_{n-2}} \ast 1_{1_{p_x} \ast (\alpha_{n-1}\cdot \alpha_n)}) \cdots (p_{\alpha_1} \ast 1_{1_{p_x} \ast (\alpha_2 \cdots \alpha_n)}) \cdot (1_{p_f} \ast (1_{1_{p_x}} \ast \phi)).
    \end{gather*}
    % }
    Again, by defining $ p_{\alpha_1 \cdots \alpha_n} $ and $ p_{\beta_1 \cdots \beta_m} $ appropriately we can rewrite this as
    \begin{align*}
        p_{\beta_1 \cdots \beta_m} \to p_{\alpha_1 \cdots \alpha_n} \cdot (1_{p_f} \ast (1_{1_{p_x}} \ast \phi)).
    \end{align*}
\end{example}

Returning to the weak higher categorical world, for low dimensional cases it is possible to construct a term $p_t $ also for the case where $ t $ is a coherence. The examples suggest that, given a term $ \Gamma \vdash t : A $ in a context, it is possible to glue the projections corresponding to the variables of $ t : A $ appropriately, and build a term $ p_t $, the type of which is of a type analogous to that in Theorem~\ref{prop:ConePsContext}. Based on this we formulate the conjecture.

\begin{conjecture}
    Let $ K $ be a cone over a diagram $ \Gamma $ with apex $ c $. Given a term $ \Gamma \vdash t : A $ of dimension $ d $, there exists a term $ K \vdash p_t : T_t $ such that $ T_t $ is given by
    \begin{align*}
        &c \to t &&\text{if } \dim(t) = 0 \\
        &p_{\tau(t)} \to p_{\sigma(t)} \mast{d}{d-1} \Big(1_{p_{\sigma^2(t)}}^{1} \mast{d}{{d}-2} \Big( 1^{2}_{p_{\sigma^{3}(t)}} \mast{d}{d-3} \dots \Big(1_{p_{\sigma^{d}(t)}}^{d - 1} \mast{d}{0} t \Big) \Big) \Big) &&\text{if } \dim(t) > 0.
    \end{align*}
    Moreover, writing $ K \equiv \Gamma, c: \Ob, \Pi $, we have
    \begin{align*}
        \FV(p_t : T_t) = \bigcup_{ \substack{p_{x_i} \in \FV(\Pi) \\ y \in \FV(t \, : \, A) \\ y \, \propto \, \tau(p_{x_i}) }} \FV(p_{x_i} : T_{x_i}).
    \end{align*}
\end{conjecture}

Assuming the conjecture we can now prove the existence of cones as contexts for all diagrams.

\begin{theorem}
    Given a context $ \Gamma $ there exists a context $ \Gamma, c : \Ob, \Pi $ such that the judgment $ \Gamma, c: \Ob, \Pi \, \textup{\texttt{cone}} \, \Gamma $ is derivable.
\end{theorem}

\begin{proof}
    This proof is similar to that of Theorem~\ref{prop:ConePsContext}.
\end{proof}

%%%%%%%%%%%%%%%%%%%%%%%%%%%%%%%%%%%%%%%%%%%%%%%%%%%%%%
%%%%%%%%%%%%%%%%%%%%%%%%%%%%%%%%%%%%%%%%%%%%%%%%%%%%%%
\subsection{The universal cone} \label{subsec:UniCone}
%%%%%%%%%%%%%%%%%%%%%%%%%%%%%%%%%%%%%%%%%%%%%%%%%%%%%%
%%%%%%%%%%%%%%%%%%%%%%%%%%%%%%%%%%%%%%%%%%%%%%%%%%%%%%

The rules of the judgment $ K \texttt{ cone } (\Gamma;c) $ allow us to generate cones as contexts for a given diagram. We now want to use this to build the universal cone. Given a diagram $ \Gamma $ and a cone $ K $ over it as a context $ K $, the universal cone is given by a collection of terms that can be assembled into a context morphism $ \Gamma \vdash \texttt{ucone} : K $.

\begin{example} \label{ex:UniCone}
    Consider for example the context $ \Gamma \equiv x : \Ob, y : \Ob, f : x \to y $.
    All the terms of the universal cone assemble into the context morphisms
    \begin{align*}
        \Gamma \vdash
    \raisebox{-1.6ex}{\resizebox{1.2\width}{3.3\height}{\Bigg \langle}}
        \begin{matrix*}
            x \\[1mm]
            y \\[1mm]
            f \\[1mm]
            \limit_\Gamma \\[1mm]
            \up_x \\[1mm]
            \up_y \\[1mm]
            \up_f
        \end{matrix*}
    \raisebox{-1.6ex}{\resizebox{1.2\width}{3.3\height}{\Bigg \rangle}} :
        \begin{pmatrix*}[l]
            x : \Ob \\[1mm]
            y : \Ob \\[1mm]
            f : x \to y \\[1mm]
            c : \Ob \\[1mm]
            p_x : c \to x \\[1mm]
            p_y : c \to y \\[1mm]
            p_f : p_y \to p_x \cdot f \\[1mm]
        \end{pmatrix*}
    \end{align*}
    Pictorially we have:
    \begin{equation} \label{diag:LimFXtoY}
        \begin{tikzcd}
            {\limit_\Gamma} \\
            & x \\
            & y
            \arrow["{\up_x}", from=1-1, to=2-2]
            \arrow[""{name=0, anchor=center, inner sep=0}, "{\up_y}"', bend right=30, from=1-1, to=3-2]
            \arrow["f", from=2-2, to=3-2]
            \arrow["{\up_f}"', shorten <=6pt, shorten >=4pt, Rightarrow, from=0, to=2-2]
        \end{tikzcd}
        \qquad\qquad\qquad\qquad
        \begin{tikzcd}
            c \\
            & x \\
            & y
            \arrow["{p_x}", from=1-1, to=2-2]
            \arrow[""{name=0, anchor=center, inner sep=0}, "{p_y}"', bend right=30, from=1-1, to=3-2]
            \arrow["f", from=2-2, to=3-2]
            \arrow["{p_f}"', shorten <=6pt, shorten >=4pt, Rightarrow, from=0, to=2-2]
        \end{tikzcd}
    \end{equation}
\end{example}
It is important to note that we explicitly build in new terms only for the cone apex and the projections in the context morphism. The cells in the context morphism corresponding to the underlying diagram are given by variables. Intuitively we think of this as allowing us to build a cone 
over a diagram of arbitrary objects and morphisms. By substitution this can be instantiated to any other specifically chosen objects and morphisms.

Unfortunately, we cannot simply spell out a rule which builds such a context morphism in one go. Due to the nature of type theory, any context morphism must be built step by step, otherwise there would be no way of accessing the terms it consists of. We are therefore forced to build the terms of the universal cone one by one with help of term constructor rules. Crucially, however, the higher cells of the universal cone depend on those of lower dimension, as dictated by the underlying diagram $ \Gamma $. As a result, when constructing a given term in the universal cone, we need to have kept track of all terms of the universal cone corresponding to the dependencies of the term at hand. We tackle this by collecting all terms at each step into a context morphism, which step by step approximates the context morphism $ \Gamma \vdash \texttt{ucone} : K $.

\paragraph{Rules for the universal cone.} The terms of the universal cone are built using the following term constructor
    \bigskip
    \begin{center}
        \begin{prooftree}
            \hypo{(\Theta,x:X,\Theta') \textup{\texttt{ cone }} (\Gamma;c)}
            \hypo{\Gamma \vdash^{\textup{\texttt{uni}}}_\Gamma \kappa : \Theta}
            \infer2{\Gamma \vdash \textup{\texttt{ucone}}_{\Gamma,x} : X[\kappa]}
        \end{prooftree}
    \end{center}
    The above term constructor refers to a new auxiliary judgment $ \Gamma \vdash^{\textup{\texttt{uni}}}_\Gamma \kappa : \Theta $ which we introduce and which is subject to the rules
    \bigskip
    \begin{center}
        \begin{prooftree}
            \hypo{\Gamma \vdash}
            \infer1{\Gamma \vdash^{\text{\textup{uni}}}_\Gamma \id_\Gamma : \Gamma}
        \end{prooftree}
    \end{center}
    \bigskip
    \begin{center}
        \begin{prooftree}
            \hypo{(\Theta,x:X,\Theta') \textup{\texttt{ cone }} (\Gamma;c)}
            \hypo{\Gamma \vdash^{\textup{\texttt{uni}}}_\Gamma \kappa : \Theta}
            \infer2{\Gamma \vdash^{\textup{\texttt{uni}}}_\Gamma \langle \kappa, \textup{\texttt{ucone}}_{\Gamma,x} \rangle : (\Theta, x : X)}
        \end{prooftree}
    \end{center}
    \bigskip
The term $ \texttt{ucone}_{\Gamma,x} $ also depends on the cone $ \Theta,x:X,\Theta' $ which should therefore also appear as a subscript in the term constructor. We suppress this for the sake of readability.

As in Example~\ref{ex:UniCone}, we will also write $ \limit_\Gamma :\equiv \texttt{ucone}_{\Gamma,c} $ where $ c : \Ob $ is the variable such that $ K \equiv \Gamma, c : \Ob, \Pi $. When no confusion can arise we will also write $ \texttt{u}x :\equiv \texttt{ucone}_{\Gamma,x} $, again as in the example.

\begin{lemma}
    Given a judgment $ \Gamma \vdash^{\textup{\texttt{uni}}}_\Gamma \kappa : \Theta $ we have
    \begin{align*}
        \Gamma \vdash, \qquad \Theta \vdash, \qquad \Gamma \vdash \kappa : \Theta.
    \end{align*}
    Moreover, the context $ \Gamma \vdash \kappa : \Theta $ is of the form $ \Gamma \vdash \langle \id_\Gamma, \limit_\Gamma, \textup{\texttt{u}}p_1, \dots, \textup{\texttt{u}}p_m \rangle : (\Gamma, c : \Ob, p_1 : T_1, \dots, p_m : T_m) $ for some $ 1 \le m \le |\Gamma| $.
\end{lemma}

\begin{proof}
    By induction.
\end{proof}

The process terminates once we have derived the judgment $ \Gamma \vdash^{\textup{\texttt{uni}}}_\Gamma \kappa : K $ where $ K $ is some cone over $ \Gamma $, producing a context morphism as in Example~\ref{ex:UniCone}. In the same way that $ K $ factorizes as $ \Gamma, c : \Ob, \Pi $, we have a corresponding factorization $ \kappa \equiv \langle \id_\Gamma, \lim_\Gamma, \overline{\texttt{u}p} \rangle $. Here we use the notation $ \overline{\texttt{u}p} $ to emphasize that we have a string of terms of the form $ \texttt{u}p $.

\begin{definition}
    If $ K $ is a cone over $ \Gamma $, given a derivable judgment $ \Gamma \vdash^{\textup{\texttt{uni}}}_\Gamma \kappa : K $ we say $ \Gamma \vdash \kappa : K $, or simply $ \kappa $, is a universal cone over $ \Gamma $ of shape $ K $.
\end{definition}

\begin{remark}
    A cone morphism contains the same information as a cone on a cone on the underlying diagram. With minor modifications it is possible to make use of the rules for cones, to generate the cells of the universal cone morphism from an arbitrary cone to the universal cone. In fact, cones might suffice to encode all higher transforms as well, making it possible to spell out the entire universal property only in terms of cones. Instead of this we will pursue another path by giving a rules which produce contexts with the shape of modifications between cones, as well as all higher transfors in a uniform way.
\end{remark}

%%%%%%%%%%%%%%%%%%%%%%%%%%%%%%%%%%%%%%%%%%%%%%%%%%%%
%%%%%%%%%%%%%%%%%%%%%%%%%%%%%%%%%%%%%%%%%%%%%%%%%%%%
\section{The universal property} \label{sec:UniProp}
%%%%%%%%%%%%%%%%%%%%%%%%%%%%%%%%%%%%%%%%%%%%%%%%%%%%
%%%%%%%%%%%%%%%%%%%%%%%%%%%%%%%%%%%%%%%%%%%%%%%%%%%%

%%%%%%%%%%%%%%%%%%%%%%%%%%%%%%%%%%%%%%%%%%%%%%%%%%%%%
%%%%%%%%%%%%%%%%%%%%%%%%%%%%%%%%%%%%%%%%%%%%%%%%%%%%%
\subsection{Gray operations} \label{subsec:Gray}
%%%%%%%%%%%%%%%%%%%%%%%%%%%%%%%%%%%%%%%%%%%%%%%%%%%%%
%%%%%%%%%%%%%%%%%%%%%%%%%%%%%%%%%%%%%%%%%%%%%%%%%%%%%

In this section we spell out a set of rules which generate the cells of the Gray tensor product of a diagram with a $d$-globe. Following Loubaton~\cite{loubaton2023theory} we refer to these constructions as Gray operations. We think of these as shapes for natural transformations and all higher transfors. The natural transformation is between two diagrams represented as contexts, one of which is the copy of the other. Note that at this stage we are only interested in generating the shape as a context. In the next section we will modify the rules so as to produce higher transfors between cones (thought of as natural transformations) and use this later to define the higher coherences of the universal cone as a context morphism.

The starting point is a diagram $ \Gamma $ given to us as a context. We then proceed in two steps: first each variable $ x \in \FV(\Gamma) $ is duplicated and second for each such variable $ x $ and its duplicate $ x' $ we append a new cell relating $ x $ and $ x' $ in a specified way. Let us showcase this with an example.

\begin{example} \label{ex:NatTrafoContext}
    Consider the diagram $ \Gamma \equiv x : \Ob, y : \Ob, f : x \to y, g : x \to y, \alpha : f \to g $. Duplication yields:
    \begin{alignat*}{5}
        \Gamma, \Gamma' \, :\equiv \ &x : \Ob, \ &&y : \Ob, \ &&f : x \to y, \ &&g : x \to y, \ &&\alpha : f \to g, \\
        &x' : \Ob, \ &&y' : \Ob, \ &&f' : x' \to y', \ &&g' : x' \to y', \ &&\alpha' : f' \to g'.
    \end{alignat*}
    Next for each variable in $ \Gamma $ and its corresponding duplicate we add a cell as follows:
    \begin{align*}
        \Gamma, \Gamma', \ &p_x : x \to x', \ p_y : y \to y', \\[1mm]
        &p_f : f \cdot p_y \to p_x \cdot f', \ p_g : g \cdot p_y \to p_x \cdot g', \\[1mm]
        &p_\alpha : (\alpha \ast 1_{p_y}) \cdot p_g \to p_f \cdot (1_{p_x} \ast \alpha').
    \end{align*}
    For later use we may also denote this context by $ \Gamma, \Gamma', P $. Pictorially we have
    \begin{center}
        \begin{tikzcd}
            x & {x'} \\
            y & {y'}
            \arrow["{{p_x}}", from=1-1, to=1-2]
            \arrow[""{name=0, anchor=center, inner sep=0}, "g", from=1-1, to=2-1]
            \arrow[""{name=1, anchor=center, inner sep=0}, "f"', bend right=60, from=1-1, to=2-1]
            \arrow["{{g'}}", from=1-2, to=2-2]
            \arrow["{{p_g}}"', Rightarrow, shorten <=10pt, shorten >=10pt, from=2-1, to=1-2]
            \arrow["{{p_y}}"', from=2-1, to=2-2]
            \arrow["\alpha", shorten <=2pt, shorten >=2pt, Rightarrow, from=1, to=0]
        \end{tikzcd}
        \qquad $ \stackrel{p_\alpha}{\equiv \mkern-2.5mu \Rrightarrow} $ 
        \begin{tikzcd}
            & x & {x'} \\
            {} & y & {y'}
            \arrow["{{p_x}}", from=1-2, to=1-3]
            \arrow["f"', from=1-2, to=2-2]
            \arrow[""{name=0, anchor=center, inner sep=0}, "{{f'}}"', from=1-3, to=2-3]
            \arrow[""{name=1, anchor=center, inner sep=0}, "{{g'}}", bend left=60, from=1-3, to=2-3]
            \arrow["{{p_f}}", Rightarrow, shorten <=10pt, shorten >=10pt, from=2-2, to=1-3]
            \arrow["{{p_y}}"', from=2-2, to=2-3]
            \arrow["{{\alpha'}}", shorten <=2pt, shorten >=2pt, Rightarrow, from=0, to=1]
        \end{tikzcd}
    \end{center}
\end{example}

\begin{example} \label{ex:ModiContext}
    Continuing Example~\ref{ex:NatTrafoContext}, let us now examine modifications. For this we begin with $ \Gamma, \Gamma', P $ and extend this context by adding the appropriate new variables. As for natural transformations we begin by duplicating certain cells. The cells we duplicate turn out to be precisely the variables in $ P $. Indeed we have
    \begin{align*}
        \Gamma, \Gamma', \ &p_x : x \to 'x, \ p_y : y \to y', \\[1mm]
        &p_f : f \cdot p_y \to p_x \cdot f', \ p_g : g \cdot p_y \to p_x \cdot g', \\[1mm]
        &p_\alpha : (\alpha \ast 1_{p_y}) \cdot p_g \to p_f \cdot (1_{p_x} \cdot \alpha'), \\[3mm]
        &p'_x : x \to x', \ \, p'_y : y \to y', \\[1mm]
        &p'_f : f \cdot p'_y \to p'_x \cdot f', \ \, p'_g : g \cdot p'_y \to p'_x \cdot g', \\[1mm]
        &p'_\alpha : (\alpha \ast 1_{p'_y}) \cdot p'_g \to p'_g \cdot (1_{p'_x} \cdot \alpha').
    \end{align*}
    As a short hand, let us denote this context by $ \Gamma, \Gamma', P, P' $ where $ P' $ contains all the duplicates of $ P $. Again, as for the natural transformation we now need to build in cells relating the variables in $ P $ with their duplicates, as follows:
    \begin{align*}
        \Gamma, \Gamma', P, P', \, & m_x : p_x \to p'_x, \ \, m_y : p_y \to p'_y \\[1mm]
        & m_f : p_f \cdot (m_x \ast 1_{f'}) \to (1_{f} \ast m_{y}) \cdot p'_f \\[1mm]
        & m_g : p_g \cdot (m_x \ast 1_{g'}) \to (1_{g} \ast m_{y}) \cdot p'_{g} \\[1mm]
        & m_\alpha : (p_\alpha \ast 1_{m_x \ast 1_{g'}}) \cdot (m_f\ast 1_{1_{p'_x}\ast\alpha'}) \to (1_{\alpha \ast 1_{p_y}} \ast m_g) \cdot (1_{1_f \ast m_y} \ast p'_\alpha). 
    \end{align*}
    which we may also abbreviate as $ \Gamma,\Gamma',P,P',M $. Diagrammatically the cell $ m_\alpha $ can be visualized as follows:
    \begin{center}
        \adjustbox{scale=0.9}
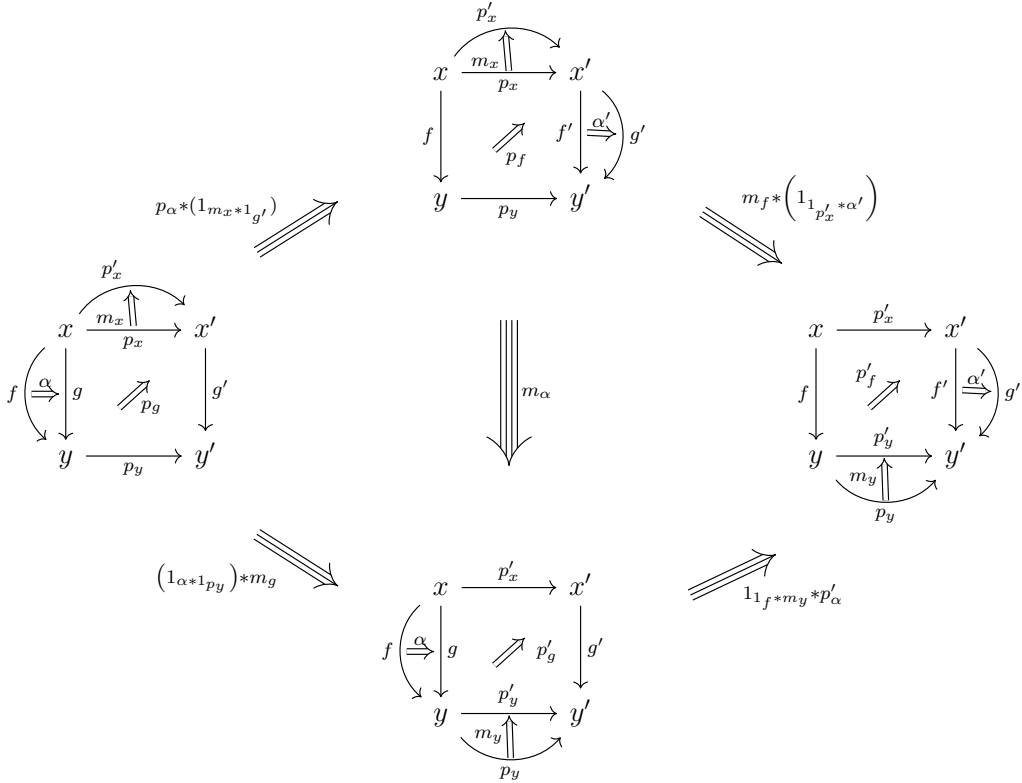
{
        \begin{tikzcd}[row sep=35, column sep=40]
            &&& x & {x'} \\
            &&& y & {y'} \\
            x & {x'} &&&&& x & {x'} \\
            y & {y'} &&&&& y & {y'} \\
            &&& x & {x'} \\
            &&& y & {y'}
            \arrow[""{name=0, anchor=center, inner sep=0}, "{{p_x}}"', from=1-4, to=1-5]
            \arrow[""{name=1, anchor=center, inner sep=0}, "{{p'_x}}", bend left=50, from=1-4, to=1-5]
            \arrow[""{name=2, anchor=center, inner sep=0}, "f"', from=1-4, to=2-4]
            \arrow[""{name=3, anchor=center, inner sep=0}, "{{f'}}"', from=1-5, to=2-5]
            \arrow[""{name=4, anchor=center, inner sep=0}, "{g'}", bend left=50, from=1-5, to=2-5]
            \arrow["{{p_f}}"'{pos=0.4}, shorten <=20pt, shorten >=20pt, Rightarrow, from=2-4, to=1-5]
            \arrow[""{name=5, anchor=center, inner sep=0}, "{{p_y}}"', from=2-4, to=2-5]
            \arrow[""{name=6, anchor=center, inner sep=0}, "{{p_x}}"', from=3-1, to=3-2]
            \arrow[""{name=7, anchor=center, inner sep=0}, "{{p'_x}}", bend left=50, from=3-1, to=3-2]
            \arrow[""{name=8, anchor=center, inner sep=0}, "g", from=3-1, to=4-1]
            \arrow[""{name=9, anchor=center, inner sep=0}, "f"', bend right=50, from=3-1, to=4-1]
            \arrow["{{g'}}", from=3-2, to=4-2]
            \arrow[""{name=10, anchor=center, inner sep=0}, "{{p'_x}}", from=3-7, to=3-8]
            \arrow["f"', from=3-7, to=4-7]
            \arrow[""{name=11, anchor=center, inner sep=0}, "{{f'}}"', from=3-8, to=4-8]
            \arrow[""{name=12, anchor=center, inner sep=0}, "{g'}", bend left=50, from=3-8, to=4-8]
            \arrow["{{p_g}}"', shorten <=20pt, shorten >=20pt, Rightarrow, from=4-1, to=3-2]
            \arrow[""{name=13, anchor=center, inner sep=0}, "{{p_y}}"', from=4-1, to=4-2]
            \arrow["{{p'_f}}", shorten <=20pt, shorten >=20pt, Rightarrow, from=4-7, to=3-8]
            \arrow[""{name=14, anchor=center, inner sep=0}, "{{p'_y}}", from=4-7, to=4-8]
            \arrow[""{name=15, anchor=center, inner sep=0}, "{{p_y}}"', bend right=50, from=4-7, to=4-8]
            \arrow[""{name=16, anchor=center, inner sep=0}, "{{p'_x}}", from=5-4, to=5-5]
            \arrow[""{name=17, anchor=center, inner sep=0}, "g", from=5-4, to=6-4]
            \arrow[""{name=18, anchor=center, inner sep=0}, "f"', bend right=50, from=5-4, to=6-4]
            \arrow[""{name=19, anchor=center, inner sep=0}, "{{g'}}", from=5-5, to=6-5]
            \arrow["{{p'_g}}"'{pos=0.7}, shorten <=20pt, shorten >=20pt, Rightarrow, from=6-4, to=5-5]
            \arrow[""{name=20, anchor=center, inner sep=0}, "{{p'_y}}", from=6-4, to=6-5]
            \arrow[""{name=21, anchor=center, inner sep=0}, "{{p_y}}"', bend right=50, from=6-4, to=6-5]
            \arrow["{{m_x}}", Rightarrow, from=0, to=1]
            \arrow["{\alpha'}", shorten <=2pt, shorten >=2pt, Rightarrow, from=3, to=4]
            \arrow["{{m_f\ast \left(1_{1_{p'_x}\ast\alpha'}\right)}}", shorten <=60pt, shorten >=50pt, Rightarrow, scaling nfold=3, from=3, to=10]
            \arrow["{{m_\alpha}}", shorten <=50pt, shorten >=50pt, Rightarrow, scaling nfold=4, from=5, to=16]
            \arrow["{p_\alpha \ast(1_{m_x \ast 1_{g'}})}", shorten <=60pt, shorten >=50pt, Rightarrow, scaling nfold=3, from=6, to=2]
            \arrow["{m_x}", shorten <=1pt, shorten >=1pt, Rightarrow, from=6, to=7]
            \arrow["\alpha", shorten <=2pt, shorten >=2pt, Rightarrow, from=9, to=8]
            \arrow["{\alpha'}", shorten <=2pt, shorten >=2pt, Rightarrow, from=11, to=12]
            \arrow["{{\left(1_{\alpha \ast 1_{p_y}}\right) \ast m_g}}"', shorten <=60pt, shorten >=50pt, Rightarrow, scaling nfold=3, from=13, to=17]
            \arrow["{{m_y}}", Rightarrow, from=15, to=14]
            \arrow["\alpha", shorten <=2pt, shorten >=2pt, Rightarrow, from=18, to=17]
            \arrow["{{1_{1_f \ast m_y} \ast p'_\alpha}}"', shorten <=50pt, shorten >=50pt, Rightarrow, scaling nfold=3, from=19, to=15]
            \arrow["{{m_y}}", Rightarrow, from=21, to=20]
        \end{tikzcd}
        }
    \end{center}
\end{example}

We now formalize this procedure in terms of type-theoretic rules producing a context of the form as in Examples~\ref{ex:NatTrafoContext} and~\ref{ex:ModiContext}. 

To keep track of the structure of these contexts we will use a semicolon instead of a comma to separate the constituent parts. The variables in between will be separated by commas as usual. With this convention, the context of Example~\ref{ex:NatTrafoContext} will be denoted by $ \Gamma; \Gamma'; P;P';M$. 

\paragraph{Rules for transfors as contexts.}
For each $ n \in \bN^> $ we introduce two new auxiliary judgment $ M_1;\dots; M_{2n+1} \ \textup{\texttt{gray}} \ \Gamma $ and $ M_1; \dots ; M_{2n+1} \ \textup{\texttt{pgray}} \ M_1 $
subject to the rules
    \bigskip
    \begin{center}
        \begin{prooftree}
            \infer0{\emptyset;\dots; \emptyset \ \textup{\texttt{gray}} \ \emptyset}
        \end{prooftree}
    \end{center}
    \bigskip
    \begin{center}
        \begin{prooftree}
            \hypo{M_1; \dots; M_{2n+1} \ \textup{\texttt{gray}} \ M_1}
            \hypo{M_1 \vdash X}
            \infer2{M_1, x :X; \dots; M_{2n+1} \ \textup{\texttt{pgray}} \ M_1,x:X}
        \end{prooftree}
    \end{center}
    \bigskip
    % \newgeometry{left=30mm,right=30mm}
    \begin{center}
        \adjustbox{scale=0.82}{
        \begin{prooftree}
            \hypo{&M_1; \dots ; M_{2i-1}, x : X; \, M_{2i}; \dots ; M_{2n+1} \ \textup{\texttt{pgray}} \ M_1}
            \infer[no rule]1{&M_1, \dots, M_{2i-1}, x : X, \, M_{2i}, x' : X, \, M_{2i+1} \vdash s \to_A t}
            \infer1{&M_1; \dots ; M_{2i-1}, x : X; \, M_{2i}, x' : X; \, M_{2i+1}, p_x : s \to_A t; \, M_{2i+2}; \dots; M_{2n+1} \ \textup{\texttt{pgray}} \ M_1}
        \end{prooftree}
        \ \footnotesize \begin{tabular}{c}
            $ 1 \le i \le n $ \\[1mm]
            $ |M_{2i-1}|=|M_{2i}| $ \\[1mm]
            $ \displaystyle{\tau\textup{Cond}(s : A, x : X, M_{2i+1} ) }  $ \\[1mm]
            $ \displaystyle{\sigma\textup{Cond}(t : A, x' : X, M_{2i+1}) }  $
        \end{tabular}
        } \normalsize
    \end{center}
    % \newgeometry{margin=40mm}
    \bigskip
    \begin{center}
        \begin{prooftree}
            \hypo{M_1; \dots; M_{2n+1} \ \textup{\texttt{pgray}} \ M_1}
            \infer1{M_1; \dots; M_{2n+1} \ \textup{\texttt{gray}} \ M_1}
        \end{prooftree}
        \qquad \footnotesize $ |M_1| = |M_{2n+1}| $ \normalsize
    \end{center}
    \bigskip

\begin{lemma} \label{lem:AdmRulesTransfors}
    The rules
    \begin{center}
        \begin{prooftree}
            \hypo{M_1;\dots;M_{2n+1} \ \textup{\texttt{gray}} \ M_1}
            \infer1{M_1,\dots,M_{2n+1} \vdash}
        \end{prooftree}
        \qquad\qquad
        \begin{prooftree}
            \hypo{M_1;\dots;M_{2n+1} \ \textup{\texttt{pgray}} \ M_1}
            \infer1{M_1,\dots,M_{2n+1} \vdash}
        \end{prooftree}
    \end{center}
    are admissible.
\end{lemma}

\begin{proof}
    By induction.
\end{proof}

If $ M_1;\dots ; M_{2n+1} \ \texttt{gray} \ M_1 $ is derivable, then we say the context $ M_1, \dots , M_{2n+1} $ is the Gray tensor product of $ M_1$ with an $n$-globe. We interpret these above rules as a recognition algorithm for those context which have the shape of a transfor over the given context.

As for cones we can formulate an existence proof for the Gray tensor product with $1$-globes. The pattern that emerges simply generalizes that of cones.

\begin{theorem}
    Let $ \Gamma $ be a globular context. There exists a context
    \begin{align*}
        T(\Gamma) \equiv \Gamma, \Gamma', M
    \end{align*}
    such that $ \Gamma;\Gamma';M \textup{\texttt{ gray }} \Gamma $ is derivable. If $ \Gamma \equiv (x_i : A_i)_{1\le i \le n} $, then $ \Gamma' $ is of the form $ (x_i' : A_i)_{1 \le i \le n} $ and $ \Pi $ of the form $ (p_{x_i} : T_{x_i} )_{1 \le i \le n} $. Moreover, the type $ T_x $ of the projection $ p_x $ of a variable $x$ of dimension $ d $ can be taken to be
    \begin{align*}
        &x \to x' && \text{if } d = 0 \\
        &\sigma(p_x) \to \tau(p_x) && \text{if } d > 0
    \end{align*}
    where
    \begin{align*}
        \sigma(p_x) = \Big( \Big( \Big( x \mast{d}{0} 1^{d-1}_{p_{\tau^{d}(x)}} \Big) \cdots \mast{d}{d - 3} 1^2_{p_{\tau^3(x)}} \Big) \mast{d}{d-2} 1^1_{p_{\tau^2(x)}} \Big) \mast{d}{d-1} p_{\tau(x)} \\[2mm]
        \tau(p_x) = p_{\sigma(x)} \mast{d}{d-1} \Big(1_{p_{\sigma^2(x)}}^{1} \mast{d}{{d}-2} \Big( 1^{2}_{p_{\sigma^{3}(x)}} \mast{d}{d-3} \dots \Big(1_{p_{\sigma^{d}(x)}}^{d - 1} \mast{d}{0} x'_i \Big) \Big) \Big).
    \end{align*}
\end{theorem}

\begin{proof}
    This can be proven in complete analogy to Proposition~\ref{prop:ConePsContext}.
\end{proof}

\begin{remark}
    A similar formula has been derived using Steiner complexes by Ara and Maltsiniotis~\cite{ara2020joint} (see Appendix B) in the context of strict $\infty$-categories.
\end{remark}

%%%%%%%%%%%%%%%%%%%%%%%%%%%%%%%%%%%%%%%%%%%%%%%%%%%%%%%
%%%%%%%%%%%%%%%%%%%%%%%%%%%%%%%%%%%%%%%%%%%%%%%%%%%%%%%
\subsection{Higher transfors between cones} \label{subsec:HiTrsfCones}
%%%%%%%%%%%%%%%%%%%%%%%%%%%%%%%%%%%%%%%%%%%%%%%%%%%%%%%
%%%%%%%%%%%%%%%%%%%%%%%%%%%%%%%%%%%%%%%%%%%%%%%%%%%%%%%

We can repeat everything done in the previous section, and apply it to cones. The rules essentially have the same form, the only difference being that the context begins with a set of variables which define a cone. 

\paragraph{Rules for higher transfors between cones.} We introduce two new auxiliary judgments $ \Gamma; c: \Ob; M_1;\dots;M_{2n+1} \ \textup{\texttt{ctrf}} \ (\Gamma;c) $ and $ \Gamma; c:\Ob, M_1;\dots;M_{2n+1} \ \textup{\texttt{pctrf}} \ (\Gamma;c) $, where $ n \in \bN^> $ is a strictly positive number, subject to the rules
    \bigskip
    \begin{center}
        \begin{prooftree}
            \infer0{\emptyset;c:\Ob; \emptyset;\dots; \emptyset \ \textup{\texttt{ctrf}} \ (\emptyset,c)}
        \end{prooftree}
    \end{center}
    \bigskip
    \begin{center}
        \begin{prooftree}
            \hypo{\Gamma; c : \Ob; M_1; \dots; M_{2n+1} \ \textup{\texttt{ctrf}} \ (\Gamma;c)}
            \hypo{\Gamma, x : X, c : \Ob, M_1, p : T \ \textup{\texttt{cone}} \ ((\Gamma,x :X),c)}
            \infer2{\Gamma,x :X; c : \Ob; M_1, p : T; M_2; \dots; M_{2n+1} \ \textup{\texttt{pctrf}} \ ((\Gamma,x:X),c)}
        \end{prooftree}
    \end{center}
    \bigskip
    \begin{center}
        \adjustbox{scale=0.8}{
        \begin{prooftree}
            \hypo{&\Gamma; c : \Ob; M_1; \dots ; M_{2i-1}, x : X; M_{2i}; \dots ; M_{2n+1} \ \textup{\texttt{pctrf}} \ (\Gamma;c)}
            \infer[no rule]1{&\Gamma; c : \Ob; M_1; \dots; M_{2i-1}, x : X; M_{2i}, x' : X; M_{2i+1} \vdash s \to_A t}
            \infer1{&\Gamma; c :\Ob; M_1; \dots ; M_{2i-1}, x : X; M_{2i}, x' : X; M_{2i+1}, p_x : s \to_A t; \dots ; M_{2n+1} \ \textup{\texttt{pctrf}} \ (\Gamma;c)}
        \end{prooftree}
        \ \footnotesize \begin{tabular}{c}
            $ 1 \le i \le n $ \\[1mm]
            $ |M_{2i-1}|=|M_{2i}| $ \\[1mm]
            $ \displaystyle{\tau\textup{Cond}\Big(s : A, x : X, \FV(M_{2i+1}) \Big) }  $ \\[1mm]
            $ \displaystyle{\sigma\textup{Cond}\Big(t : A, x' : X, \FV(M_{2i+1})\Big) }  $
        \end{tabular}
        } \normalsize
    \end{center}
    \bigskip
    \begin{center}
        \begin{prooftree}
            \hypo{\Gamma; c : \Ob; M_1; \dots; M_{2n+1} \ \textup{\texttt{pctrf}} \ (\Gamma;c)}
            \infer1{\Gamma; c : \Ob; M_1; \dots; M_{2n+1} \ \textup{\texttt{ctrf}} \ (\Gamma;c)}
        \end{prooftree}
        \qquad \footnotesize $ |M_1| = |M_{2n+1}| $ \normalsize
    \end{center}
    \bigskip
% \end{definition}

If the judgment $ \Gamma;c : \Ob; M_1;\dots;M_{2n+1} \ \texttt{ctrf} \ (\Gamma;c) $ is derivable we say that $ \Gamma;c : \Ob; M_1;\dots;M_{2n+1} $ is a conical $(n+1)$-transfor over $ \Gamma $. A conical $2$-transfor over $\Gamma$ is a modification of cones over $\Gamma$.

The lists of variables generated by the above rules are contexts, as the following lemma affirms.

\begin{lemma} \label{lem:AdmRulesConeTransfors}
    The rule
    \begin{center}
        \begin{prooftree}
            \hypo{\Gamma; c : \Ob; M_1; \dots; M_{2n+1} \ \textup{\texttt{ctrf}} \ (\Gamma;c)}
            \infer1{\Gamma, c : \Ob, M_1, \dots, M_{2n+1} \vdash }
        \end{prooftree}
    \end{center}
    is admissible.
\end{lemma}

\begin{proof}
    By induction.
\end{proof}

\begin{lemma} \label{lem:ConeGrayStructure}
    Given a derivable judgment $ \Gamma; c : \Ob; M_1; \dots; M_{2n+1} \ \textup{\texttt{ctrf}} \ (\Gamma, c) $ we have $ |\Gamma| = |M_i| $ for all $ 1 \le i \le 2n +1 $. Moreover $ K :\equiv \Gamma, c : \Ob, M_1 $ is a cone over $ \Gamma $ with apex $ c $.
\end{lemma}

In some cases, instead of the full $ \Gamma; c : \Ob; M_1; \dots; M_{2n+1} \ \texttt{ctrf} \ (\Gamma;c)$, we will also write $ M \ \texttt{ctrf} \ (\Gamma;c) $. In such situation we will write $ n_M \in \bN^> $ instead of $ n $, to make clear which natural number we are referring to.

%%%%%%%%%%%%%%%%%%%%%%%%%%%%%%%%%%%%%%%%%%%%%%%%%
%%%%%%%%%%%%%%%%%%%%%%%%%%%%%%%%%%%%%%%%%%%%%%%%%
\subsection{Composing with cones} \label{subsec:Whiskering}
%%%%%%%%%%%%%%%%%%%%%%%%%%%%%%%%%%%%%%%%%%%%%%%%%
%%%%%%%%%%%%%%%%%%%%%%%%%%%%%%%%%%%%%%%%%%%%%%%%%

Let $ \Gamma $ be a diagram and assume we are given a universal cone $ \Gamma \vdash \texttt{ucone} : K $ with apex $\lim_\gamma : \Ob$. Given a morphism of type $ c \to \limit_\Gamma $, where $c:\Ob$ is some variable, we can compose this with $ \texttt{ucone} $ and form a cone with apex $ c $. More generally, given any cell in the $(\infty,\infty)$-category of morphisms given by the type $ c \to \limit_\Gamma $, by composition we get a term in the $(\infty,\infty)$-category of cones over $ \Gamma $ with apex $ c $. Schematically there exists a functor of $ (\infty,\infty)$-categories
\begin{align*}
    \texttt{ucone}_* : \{ \text{terms of } c \to \limit_\Gamma \} \longrightarrow \{ \text{cones } K \text{ over } \Gamma \text{ with apex } c \}.
\end{align*}
The cone $ \texttt{ucone} $ is a limiting cone if this functor is an equivalence.

Spelling out this universal property requires us to be able to refer to the image of this functor. We encode this again by a context morphism.

Let $ \Gamma $ be a context and let $ \Gamma \vdash \texttt{ucone} : K $ be a context morphism encoding a universal cone obtained from the rules in Subsection~\ref{subsec:UniCone}. The cone will be of the form $ K \equiv \Gamma, c : \Ob, \Pi $ where $ \Pi $ contains all the projection variables. The context morphism $\texttt{ucone}$ can then be split accordingly and we denote by $ \overline{\texttt{u}p} $ the list of terms in $ \texttt{ucone} $ corresponding to $ \Pi $. Now, given a context $ \Gamma, c : \Ob, f : c \to \limit_\Gamma $, applying $ \texttt{ucone}_* $ to $ f $ gives a context morphism schematically denoted by
\begin{align*}
    \Gamma, c : \Ob, f : c \to \limit_\Gamma \vdash \langle \id_\Gamma, c, f \ast \overline{\texttt{u}p} \rangle : K.
\end{align*}
Here $ f \ast \overline \up $ stands for a list of terms of the same length as $ \Gamma \equiv (x_i : A_i)_{1\le i\le n} $, the $i$-th term of which is given by composing $ \up_{x_i} $ with $ f $ appropriately. 

More generally, let $ M $ be a conical $n$-transfor over $ \Gamma $. If $ \Gamma, c : \Ob, D^n(c,\limit_\Gamma,\phi) $ is a context extending $ \Gamma $ where $ D^n(c,c',\phi) $ is a 
\begin{align*}
    D^n(c,\limit_\Gamma, \phi) :\equiv f : c \to \limit_\Gamma, \ g : c \to \limit_\Gamma, \alpha : f \to g, \beta : f \to g, \dots, \phi : A_\phi
\end{align*}
has the shape of an $n$-dimensional globe between $ \limit_\Gamma $ and some variable $ c : \Ob $ with maximal cell $ \phi $, we would like to obtain a context morphism of the form
\begin{align*}
    \Gamma, c : \Ob, D^n(c,\limit_\Gamma, \phi) \vdash \langle \id_\Gamma, c, f \ast \overline \up, \ g \ast \overline \up, \ \dots \ , \ \phi \ast \overline \up \rangle : M. 
\end{align*}
Up to a manageable dimension we can depict this pictorially as
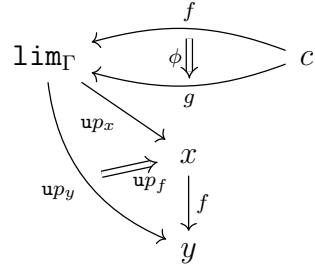
\begin{center}
    \begin{tikzcd}
        {\limit_\Gamma} && c \\
        & x \\
        & y
        \arrow["{\up_x}"', from=1-1, to=2-2]
        \arrow[""{name=0, anchor=center, inner sep=0}, "{\up_y}"', bend right=30, from=1-1, to=3-2]
        \arrow[""{name=1, anchor=center, inner sep=0}, "f"', bend right=18, from=1-3, to=1-1]
        \arrow[""{name=2, anchor=center, inner sep=0}, "g", bend left=18, from=1-3, to=1-1]
        \arrow["f", from=2-2, to=3-2]
        \arrow["{\up_f}"', shorten <=6pt,  shorten >=3pt, Rightarrow, from=0, to=2-2]
        \arrow["\phi"', shorten <=3pt, shorten >=2pt, Rightarrow, from=1, to=2]
    \end{tikzcd}
\end{center}
In the above diagram, composing the universal cone with $f $ and $ g $ respectively gives two cones. Composing with $ \phi $ produces a modification between these two cones.

To obtain such a context morphism, we use the same trick as before and work with an arbitrary cone in the form of a context instead of the universal cone. This allows us to control which projections contribute in which terms by fixing the free variables. In a bit more detail, assume we are given a judgment $ K \texttt{ cone } \Gamma $, that is $ K $ is a cone over $ \Gamma $. It will necessarily be of the form $ K \equiv \Gamma, c : \Ob, \Pi $, where $ \Pi $ contains all the projection variables. Then, what we are looking for a context morphism which schematically is of the form
\begin{align*}
    K, c' : \Ob, D^n(c',c) \vdash \langle \id_\Gamma, c, f \ast \id_\Pi, g \ast \id_\Pi, \ \dots \ , \phi \ast \id_\Pi \rangle : M
\end{align*}
where $ \id_\Pi $ denotes the variables $ \FV(\Pi) $ as a sequence of terms with the induced order. In the next subsection we give a set of rules which produce such context morphisms.

%%%%%%%%%%%%%%%%%%%%%%%%%%%%%%%%%%%%%%%%%%%%%%%%%%%%
%%%%%%%%%%%%%%%%%%%%%%%%%%%%%%%%%%%%%%%%%%%%%%%%%%%%
\subsection{Composing with cones and the universal property} \label{subsec:WhiskUP}
%%%%%%%%%%%%%%%%%%%%%%%%%%%%%%%%%%%%%%%%%%%%%%%%%%%%
%%%%%%%%%%%%%%%%%%%%%%%%%%%%%%%%%%%%%%%%%%%%%%%%%%%%

Let $ \Gamma \vdash \kappa : K $ be a universal cone over a diagram $ \Gamma $. As explained in subsection~\ref{subsec:Whiskering}, the idea is to build into \texttt{CaTT} the required terms to make the functor
\begin{align*}
    \texttt{ucone}_* : \{ \text{terms of } c \to \limit_\Gamma \} \longrightarrow \{ \text{cones } K \text{ over } \Gamma \text{ with apex } c \}
\end{align*}
into an equivalence of $ (\infty,\infty)$-categories. Switching momentarily to non type-theoretic notation, a functor $ F : \mathcal C \to \mathcal D $ of $(\infty,\infty)$-categories is said to be an equivalence if the maps
\begin{enumerate}[label=(\roman*)]
    \item $ F : \mcC_0 \to \mcD_0 $;
    \item $ F_{x,y} : \Hom_\mcC(x,y) \to \Hom_\mcD(Fx,Fy) $ $ x,y \in \mcC $ for all $ x,y \in \mcC $;
    \item $ F_{f,g} : \Hom_{\Hom(x,y)}(f,g) \to \Hom_{\Hom(Fx,Fy)}(Ff,Fg) $ for all $ f,g \in \Hom(x,y) $;
    \item etc.
\end{enumerate}
are surjective.

By surjectivity we mean, essential surjectivity, i.e.~surjectivity up to equivalence. By equivalence we mean coinductive equivalence. Coinductive equivalences have been studied by Benjamin and Markakis~\cite{benjamin2024invertible} in the related framework of computads due to Dean et al.~\cite{dean2024computads}. Note that every coherence is a coinductive equivalence.

\paragraph{Rules for invertible morphisms.}
Given a term $ \Gamma \vdash u : s \to t $ in \texttt{CaTT}, we can ensure this is an equivalence by introducing a new judgment $ \Gamma \vDash u : s \to t $ which obeys the rules
\medskip
\begin{center}
    \begin{prooftree}
        \hypo{\Gamma \vDash u : s \to t}
        \infer1{\Gamma \vdash u : s \to t \qquad\qquad \Gamma \vdash \texttt{inv}(u) : t \to s}
    \end{prooftree}
\end{center}
\medskip
\begin{center}\begin{prooftree}
    \hypo{\Gamma \vDash u : s \to t}
    \infer1{\Gamma \vDash \texttt{eta}(u) : 1_s \to u \cdot \texttt{inv}(u) \qquad\qquad \Gamma \vDash \texttt{eps}(u) : \texttt{inv}(u) \cdot u \to 1_t}
\end{prooftree}
\end{center}
\medskip
We say that an $n$-transfor is invertible, if all of its components are invertible.

Returning to $ \texttt{ucone}_* $, on objects we must require the following: given a universal cone $ \Gamma \vdash \texttt{ucone} : K $ over $ \Gamma $, where $ K \equiv (\Gamma, c : \Ob, \Pi)$, there exists a term $ \text{u}m : c \to \limit_\Gamma $ and an invertible modification $ \overline{\texttt{u}m} $ between the cone $ K $ and the universal cone composed with $\texttt{u}m $. Working again with our example $ \Gamma \equiv x : \Ob, y : \Ob, f : x \to y $, pictorially we have:

\begin{equation}  \label{diag:UniCone}
    \begin{tikzcd}
        {\lim_\Gamma} && c \\
        & x \\
        & y
        \arrow["{\texttt{u}p_x}"', from=1-1, to=2-2]
        \arrow[""{name=0, anchor=center, inner sep=0}, "{\texttt{u}p_y}"', bend right=30, from=1-1, to=3-2]
        \arrow["{\texttt{u}m}"', from=1-3, to=1-1]
        \arrow[""{name=1, anchor=center, inner sep=0}, "{p_x}", from=1-3, to=2-2]
        \arrow[""{name=2, anchor=center, inner sep=0}, "{p_y}", bend left=30, from=1-3, to=3-2]
        \arrow["f", from=2-2, to=3-2]
        \arrow["{\texttt{u}p_f}"', shorten <=4pt, Rightarrow, from=0, to=2-2]
        \arrow["\simeq", shorten <=12pt, shorten >=15pt, Rightarrow, from=1, to=1-1]
        \arrow["{p_f}", shorten <=4pt, Rightarrow, from=2, to=2-2]
    \end{tikzcd}
    \qquad\qquad
    \begin{tikzcd}
        {\lim_\Gamma} && c \\
        \\
        & y
        \arrow["{{\texttt{u}p_y}}"', bend right=30, from=1-1, to=3-2]
        \arrow["{{\texttt{u}m}}"', from=1-3, to=1-1]
        \arrow[""{name=0, anchor=center, inner sep=0}, "{{p_y}}", bend left=30, from=1-3, to=3-2]
        \arrow["\simeq", shorten <=20pt, shorten >=15pt, Rightarrow, from=0, to=1-1]
    \end{tikzcd}
\end{equation}

If $ M \equiv (\Gamma; c : \Ob; M_{1}; M_{2}; M_{3}) $ is a modification of cones, that is $ M \texttt{ ctrf } (\Gamma;c) $, then the whole data of diagram~\ref{diag:UniCone} can be organized into the context morphism
\begin{align*}
    K \vdash \langle \id_\Gamma, c, \id_\Pi, \texttt{u}m \ast \overline{ \texttt{u}p}, \overline{ \texttt{u}m }\rangle : (\Gamma, c : \Ob, M_{1}, M_{2}, M_{3}).
\end{align*}
Here, if $\Gamma \equiv (x_i : A_i)_{1\le i \le n}$,  $\texttt{u}m \ast \overline m$ denotes a list of $|\Gamma|$ terms, the $i$-th term of which is given by composing $ \texttt{u}m $ with $\texttt{u}p_{x_i}$ appropriately. So in the above example we have $ \texttt{u}m \ast \overline{\texttt{u}p} \equiv (\texttt{u}m \cdot\texttt{u}p_x, \texttt{u}m\cdot\texttt{u}p_y, 1_{\texttt{u}m}\mast{2}{0}\texttt{u}p_f) $. The notation $\overline{\texttt{u}m}$, on the other hand, contains all the components of the modification. In the above example this includes the two invertible cells, depicted in the diagram, as well as a $3$-dimensional cell interpolating between the two diagrams.

At the next level we are first given two maps $ f, g : c \to \limit_\Gamma $ with which we obtain two cones $ f \ast \overline{\texttt{u}p} $ and $ g \ast \overline{\texttt{u}p} $ with apex $ c $ by composition. Then, given any modification between these two cones, namely a list of variables denoted by the shorthand notation $ \overline m :  f \ast \overline{\texttt{u}p} \to g \ast \overline{\texttt{u}p} $, there should exist a cell $ \texttt{u}q : f \to g $, which when composed with the universal cone yields an modification $ \texttt{u}q \ast \overline{\texttt{u}p} $ equivalent to $ \overline m $. Using again a shorthand notation we denote this invertible map by $ \overline{\texttt{u}q} : \overline m \to \texttt{u}q \ast \overline{\texttt{u}p} $. As a context morphism this is given by
\begin{align} \label{eq:WhiskContMor}
    \Gamma, \, c : \Ob, \, f : c \to \limit_\Gamma, \, g : c \to \limit_\Gamma, \, \overline m : f \ast \overline{\texttt{u}p} \to g \ast \overline{\texttt{u}p} \vdash 
    \raisebox{-1.6ex}{\resizebox{1.2\width}{3.3\height}{\Bigg \langle}}
        \begin{matrix*}
            \id_\Gamma \\[1mm]
            c \\[1mm]
            f \ast \overline{\texttt{u}p} \\[1mm]
            g \ast \overline{\texttt{u}p} \\[1mm]
            \overline m \\[1mm]
            \texttt{u}q \ast \overline{\texttt{u}p} \\[1mm]
            \overline{\texttt{u}q} 
        \end{matrix*}
    \raisebox{-1.6ex}{\resizebox{1.2\width}{3.3\height}{\Bigg \rangle}} :
    \begin{pmatrix*}
        \Gamma \\[1mm]
        c : \Ob \\[1mm]
        M_{1} \\[1mm]
        M_{2} \\[1mm]
        M_{3} \\[1mm]
        M_{4} \\[1mm]
        M_{5}
    \end{pmatrix*}
\end{align}
where $ M \equiv (\Gamma; c : \Ob; M_{1}; \dots; M_{5}) $ encodes a perturbation between modifications of cones. To obtain such a context morphism, we will work with an arbitrary cone in the form of a context $ K $ over $ \Gamma $ (instead of the universal cone) where we can control the variables, and subsequently substitute a universal cone into the construction. As an example, the context morphism in equation~\ref{eq:WhiskContMor} is obtained as the composite of the two context morphisms
\begin{gather}
    \begin{pmatrix*} \label{eq:FirstContextMorphism}
        \Gamma \\[1mm]
        c : \Ob \\[1mm]
        f : c \to \limit_\Gamma \\[1mm]
        g : c \to \limit_\Gamma \\[1mm]
        \overline m : f \ast \overline{\texttt{u}p} \to g \ast \overline{\texttt{u}p} \\[1mm]
    \end{pmatrix*} \vdash
    \raisebox{-1.6ex}{\resizebox{1.2\width}{3.3\height}{\Bigg \langle}}
        \begin{matrix*}
            \kappa \\[1mm]
            c \\[1mm]
            f \\[1mm]
            g \\[1mm]
            \overline{m} \\[1mm]
            \texttt{u}q \\[1mm]
            \overline{\texttt{u}q} 
        \end{matrix*}
    \raisebox{-1.6ex}{\resizebox{1.2\width}{3.3\height}{\Bigg \rangle}} :
    \begin{pmatrix*}
        K \\[1mm]
        c' : \Ob \\[1mm]
        f' : c' \to c \\[1mm]
        g' : c' \to c \\[1mm]
        \Delta \\[1mm]
        \alpha' : f' \to g' \\[1mm]
        \Delta'
    \end{pmatrix*} \\[5mm] \label{eq:SecondContextMorphism}
    \begin{pmatrix*}
        K \\[1mm]
        c' : \Ob \\[1mm]
        f' : c' \to c \\[1mm]
        g' : c' \to c \\[1mm]
        \Delta \\[1mm]
        \alpha' : f' \to g' \\[1mm]
        \Delta'
    \end{pmatrix*} \vdash 
    \raisebox{-1.6ex}{\resizebox{1.2\width}{3.3\height}{\Bigg \langle}}
        \begin{matrix*}
            \id_\Gamma \\[1mm]
            c' \\[1mm]
            f' \ast \id_\Pi \\[1mm]
            g' \ast \id_\Pi \\[1mm]
            \id_\Delta \\[1mm]
            \alpha' \ast \id_{\Pi} \\[1mm]
            \id_{\Delta'}
        \end{matrix*}
    \raisebox{-1.6ex}{\resizebox{1.2\width}{3.3\height}{\Bigg \rangle}} :
    \begin{pmatrix*}
        \Gamma \\[1mm]
        c : \Ob \\[1mm]
        M_{1} \\[1mm]
        M_{2} \\[1mm]
        M_{3} \\[1mm]
        M_{4} \\[1mm]
        M_{5}
    \end{pmatrix*}
\end{gather}
The order in which the variables appear in the context $ K,c': \Ob, f' : c' \to c, g' : c' \to c, \Delta, \alpha' : f' \to g', \Delta' $ guarantees that we can later associate the correct free variables to the term $ \texttt{u}q $ and those in the list of terms $ \overline{\texttt{u}q} $.

The context morphisms are built inductively by moving step by step through all the variables of $ M $. Even though it is possible to do without, let us introduce a new judgment $ \Delta \vdash_\Gamma x : A$, which we will use as a convenient device to keep track of the process while keeping the notation more compact. This judgment is subject to the rules
    \begin{center}
        \begin{prooftree}
            \hypo{\Gamma \vdash A}
            \infer1{\Gamma, x : A \vdash_{\Gamma} x : A}
        \end{prooftree}
        \qquad\qquad
        \begin{prooftree}
            \hypo{\Delta \vdash_\Gamma x : A}
            \hypo{\Delta \vdash B}
            \infer2{\Delta, y : B \vdash_\Gamma x : A.}
        \end{prooftree}
    \end{center}
    \bigskip
    If $ \Delta \vdash_\Gamma x : A $ is derivable, we say $ \Delta $ is an extension of $ \Gamma $.

\begin{lemma} \label{lem:ExtensionRules}
    Given the judgement $ \Delta \vdash_\Gamma x : A $, then the following judgments are derivable:
    \begin{align*}
        \Gamma \vdash, \qquad \Delta \vdash, \qquad \Delta \vdash x : A.
    \end{align*}
    Moreover, $ \Delta \equiv \Gamma, x : A, \Delta' $ for some list of variables $ \Delta' $.
\end{lemma}

\begin{proof}
    By induction.
\end{proof}

We now spell out a set of rules which generate the context morphism of the form of equation~\ref{eq:SecondContextMorphism}. The first context morphism of equation~\ref{eq:FirstContextMorphism}, on the other hand, is what will be generated by the rules for the universal property in the next subsection.

\paragraph{Rules for postcomposition with a cone.}
    We introduce a new judgment $ W \vdash^\whisk_{K; \alpha} w : M $ subject to the rules
    \begin{center}
        \adjustbox{scale=0.8}{
        \begin{prooftree}
            \hypo{K \textup{\texttt{ cone }} (\Gamma, c)}
            \infer1{K, c' : \Ob, D^{n}(c,c',\alpha) \vdash^\textup{\pstar}_{K; \alpha } \langle \id_\Gamma, c' \rangle : (\Gamma, c : \Ob)}
        \end{prooftree}
        }
    \end{center}
    \bigskip
    \begin{center}
        \adjustbox{scale=0.8}{
        \begin{prooftree}
            \hypo{K \textup{\texttt{ cone }} (\Gamma, c)}
            \infer[no rule]1{M \textup{\texttt{ ctrf }} (\Gamma, c)}
            \hypo{W \vdash^\textup{\pstar}_{K;\alpha } w : \Theta}
            \hypo{M \vdash_\Theta x : X}
            \hypo{\Theta \vdash u : X[w]}
            \infer4{W \vdash^\textup{\pstar}_{K; \alpha } \langle w, u \rangle : (\Theta, x : X)}
        \end{prooftree}
        \quad
        \begin{tabular}{c}
            $\dim(\alpha) = n_M $\\[2mm]
            $x \in \FV(M_{j}), \ j \in \{1, \dots, 2n_M - 2, 2n_M \} $ \\[2mm]
            $ \textup{Star}(\Gamma, K, M, x, u : X[w], \alpha : A) $
        \end{tabular}
        }
    \end{center}
    \bigskip
    \begin{center}
        \adjustbox{scale=0.8}{
        \begin{prooftree}
            \hypo{K \textup{\texttt{ cone }} (\Gamma, c)}
            \infer[no rule]1{M \textup{\texttt{ ctrf }} (\Gamma, c)}
            \hypo{W \vdash^\textup{\pstar}_{K;\alpha } w : \Theta}
            \hypo{M \vdash_\Theta x : X}
            \infer3{W, x' : X[w] \vdash^\textup{\pstar}_{K;\alpha } \langle w, x' \rangle : (\Theta, x : X)}
        \end{prooftree}
        \quad
        \begin{tabular}{c}
            $\dim(\alpha) = n_M $\\[2mm]
            $x \in \FV(M_{j}), \ j \in \{ 2n_M - 1, 2n_M +1 \} $
        \end{tabular}
        }
    \end{center}
    \bigskip
    \begin{center}
        \adjustbox{scale=0.8}{
        \begin{prooftree}
            \hypo{W, \alpha : A, x : X, W' \vdash^\textup{\pstar}_{K,\alpha } w : \Theta}
            \infer1{W, x : X, \alpha : A, W' \vdash^\textup{\pstar}_{K;\alpha } w : \Theta}
        \end{prooftree}
        \quad
        \begin{tabular}{c}
            $ \alpha \not \in \FV(X) $
        \end{tabular}
        }
    \end{center}
    \bigskip
    \begin{center}
        % \qquad\quad
        \adjustbox{scale=0.8}{
        \begin{prooftree}
            \hypo{W, \alpha : A, x : X, W' \vdash^\textup{\pstar}_{K,\alpha } w : \Theta}
            \infer1{W, \alpha : A, x : X, W' \vdash^\textup{\whisk}_{K;\alpha } w : \Theta}
        \end{prooftree}
        \quad
        \begin{tabular}{c}
            $ \alpha \in \FV(X) $
        \end{tabular}
        }
    \end{center}
    \bigskip
    where, with the notation $ \Gamma \equiv (x_i : A_i)_{1 \le i \le l} $ and $ K \equiv (\Gamma, c : \Ob, (p_i : T_{x_i})_{1\le i \le l}) $, as well as $ M_{j} \equiv (x_{j,i} : T_{j,i})_{1 \le i \le l} $, the shorthand $ \textup{Star}(\Gamma,K,M,x_{j,i}, u : X, \alpha : T_\alpha) $ stands for
    \begin{align*}
        \FV(u : X) &= \FV(\alpha : T_\alpha) \cup \FV(p_i : T_i) \\[3mm]
        \FV(u : X) &= \begin{cases}
            \FV(\sigma^{n-j}(\alpha) : \partial T_\alpha ) \cup \FV(p_i : T_{x_i}), \ & i \text{ odd} \\
            \FV(\tau^{n-j}(\alpha) : \partial T_\alpha )\cup \FV(p_i : T_{x_i}), \ & i \text{ even}
        \end{cases}
    \end{align*}

\begin{lemma}
    Given a derivable judgment $ W \vdash^\whisk_{K;\alpha} w : \Theta $, the following judgments are derivable
    \begin{align*}
        W \vdash, \qquad \Theta \vdash, \qquad W \vdash w : \Theta.
    \end{align*}
\end{lemma}

\begin{proof}
    By induction.
\end{proof}

%%%%%%%%%%%%%%%%%%%%%%%%%%%%%%%%%%%%%%%%%%%%%%%%%%%%%
%%%%%%%%%%%%%%%%%%%%%%%%%%%%%%%%%%%%%%%%%%%%%%%%%%%%%
\subsection{The universal property} \label{subsec:UP}
%%%%%%%%%%%%%%%%%%%%%%%%%%%%%%%%%%%%%%%%%%%%%%%%%%%%%
%%%%%%%%%%%%%%%%%%%%%%%%%%%%%%%%%%%%%%%%%%%%%%%%%%%%%

We are now ready to formulate the universal property.

\paragraph{Rules for the universal property.}
    The terms generated by the universal property are obtained by the term constructor rules
    \medskip
    \begin{center}
        \adjustbox{scale=0.87}{
            \begin{prooftree}
                \hypo{& K \textup{\texttt{ cone }} (\Gamma;c)}
                \infer[no rule]1{& M \textup{\texttt{ ctrf }}(\Gamma, c)}
                \hypo{&\Lambda, \alpha : T_\alpha, \Delta \vdash_{K, \Theta} x : X}
                \infer[no rule]1{&\Lambda, \alpha : T_\alpha, \Delta \vdash^\whisk_{K;\alpha } w : M}
                \hypo{\Omega \vdash^{\textup{\texttt{uni}}}_\Gamma \theta : (K,\Theta)}
                \infer3{J_1 \qquad J_2}
            \end{prooftree}
            \qquad
            \begin{tabular}{c}
                $ \dim(\alpha) = n_M $ %\\[2mm]
                % $ \alpha \in \FV(Y) \textup{ for all } (y : Y) \in \Delta $
            \end{tabular}
        }
    \end{center}
    \bigskip
    where the $J_1,J_2$ together with additional side conditions are given by:
    \begin{align*}
        J_1: && &\Omega \vdash \textup{\texttt{uni}}_{\Gamma,x} : X[\theta] && x \in \{\alpha\} \\[2mm] 
        J_2: && &\Omega \vDash \textup{\texttt{uni}}_{\Gamma,x} : X[\theta] && x \in \FV(\Delta).
    \end{align*}
    These term constructor rules refer to a new judgment $ \Omega \vdash^{\texttt{uni}}_\Gamma \theta : \Theta $, generalizing the judgment $ \Gamma \vdash_\Gamma^{\texttt{uni}} \kappa : \Theta$ introduced in the rules for the universal cone, and which is subject to the rules
    \bigskip
    \begin{center}
        \adjustbox{scale=0.87}{
            \begin{prooftree}
                \hypo{& K \textup{\texttt{ cone }} (\Gamma;c)}
                \infer[no rule]1{& M \textup{\texttt{ ctrf }}(\Gamma, c)}
                \hypo{&\Lambda, \alpha : T_\alpha, \Delta \vdash_{K, \Theta} x : X}
                \infer[no rule]1{&\Lambda, \alpha : T_\alpha, \Delta \vdash^\whisk_{K;\alpha } w : M}
                \hypo{\Omega \vdash^{\textup{\texttt{uni}}}_\Gamma \theta : (K,\Theta)}
                \infer3{J_3 \qquad J_4}
            \end{prooftree}
            \qquad
            \begin{tabular}{c}
                $ \dim(\alpha) = n_M $
            \end{tabular}
        }
    \end{center}
    \begin{align*}
        J_3: && &\Omega, x' : X[\theta] \vdash^{\textup{\texttt{uni}}}_\Gamma \langle \theta, x' \rangle : (\Omega, x : X) && x \in \FV(\Lambda) \\[2mm]
        J_4: && &\Omega \vdash^{\textup{\texttt{uni}}}_\Gamma \langle \theta, \textup{\texttt{uni}}_{\Gamma,x} \rangle : (\Theta, x : X) && x \in \{\alpha\} \cup \FV(\Delta).
    \end{align*}
% \end{definition}
The terms $ \texttt{uni}_{\Gamma,x} $ also depend on $ M $ and $K $ as well as the context used in the \texttt{star} judgment, which should therefor also appear in the term constructor as subscripts. This, however would make the notation unwieldy, because of which we suppress these dependencies.

\begin{example} \label{ex:TerminalObject}
    Consider the empty diagram $ \emptyset \vdash $. Applying the rules for cones we can derive the judgment $ (c : \Ob) \texttt{ cone } (\emptyset,c) $, that is, a cone over the empty diagram is just a 0-cell, the apex. A universal cone over $ \emptyset $ of shape $ c : \Ob $ is then given by a context morphism $ \emptyset \vdash^{\texttt{uni}}_\emptyset \langle \limit_\emptyset \rangle : (c : \Ob) $. Let us write $ \top $ instead of $ \limit_\emptyset $.

    Regarding the universal property, a conical 2-transfor over $ \emptyset $ is given by the context $ c : \Ob $, or more precisely by the judgment $ (\emptyset;c : \Ob;\emptyset,\emptyset,\emptyset) \texttt{ ctrf } (\emptyset,c) $. The \texttt{star} judgement required by the rule must be of the form $ c : \Ob, c' : \Ob, f : c' \to c \vdash \langle c' \rangle : (c : \Ob) $. The only judgment of the form $ \vdash^\texttt{uni}_\emptyset $ available is $ \emptyset \vdash^{\texttt{uni}} \langle \top \rangle : (c : \Ob) $. Putting everything together, the first application of the rule gives
    \bigskip
    \begin{center}
        \adjustbox{scale=0.9}{
        \begin{prooftree}
            \hypo{&(c : \Ob) \texttt{ cone } (\emptyset,c) }
            \infer[no rule]1{& (c : \Ob) \texttt{ ctrf } (\emptyset,c)}
            \hypo{&c : \Ob, c' : \Ob, f : c' \to c \vdash_{c : \Ob} c': \Ob }
            \infer[no rule]1{&c : \Ob, c' : \Ob, f : c' \to c \vdash^\whisk_{c : \Ob, f : c' \to c} \langle c' \rangle : (c : \Ob)}
            \hypo{&\emptyset \vdash^{\texttt{uni}}_\emptyset \langle \top \rangle : (c : \Ob)}
            \infer3{ c' : \Ob \vdash^{\texttt{uni}}_\emptyset \langle \top, c' \rangle : (c : \Ob, c' : \Ob) }
        \end{prooftree}
        }
    \end{center}
    \bigskip
    and the second and final application gives
    \bigskip
    \begin{center}
        \adjustbox{scale=0.8}{
        \begin{prooftree}
            \hypo{&(c : \Ob) \texttt{ cone } (\emptyset,c) }
            \infer[no rule]1{& (c : \Ob) \texttt{ ctrf } (\emptyset,c)}
            \hypo{&c : \Ob, c' : \Ob, f : c' \to c \vdash_{c : \Ob, c': \Ob} f : c' \to c }
            \infer[no rule]1{&c : \Ob, c' : \Ob, f : c' \to c \vdash^\whisk_{c : \Ob, f : c' \to c} \langle c' \rangle : (c : \Ob)}
            \hypo{&c' : \Ob \vdash^{\texttt{uni}}_\emptyset \langle \top, c' \rangle : (c : \Ob, c' : \Ob)}
            \infer3{c' : \Ob \vdash^{\texttt{uni}}_\emptyset \langle \top, c', \texttt{uni}_{\emptyset,f} \rangle : (c : \Ob, c' : \Ob, f : c' \to c)}
        \end{prooftree}
        }
    \end{center}
    \bigskip
    Pictorially we have
    \begin{center}
        \begin{tikzcd}
            {\top} & {c'}
            \arrow["{\texttt{uni}_{\emptyset,f}}", from=1-2, to=1-1]
        \end{tikzcd}
    \end{center}
\end{example}

\begin{example}
    Consider again the diagram $ \emptyset \vdash $. This time we consider $ (c : \Ob) $ as a 3-transfor over $ \emptyset $. Application of the rule then gives
    \begin{center}
        \begin{prooftree}
            \hypo{&(c : \Ob) \texttt{ cone } (\emptyset,c) }
            \infer[no rule]1{& (c : \Ob) \texttt{ ctrf } (\emptyset,c)}
            \infer[no rule]1{&c : \Ob, c' : \Ob, f : c' \to c, g : c' \to c, \alpha : f \to g \vdash_{c : \Ob} c' : \Ob }
            \infer[no rule]1{&c : \Ob, c' : \Ob, f : c' \to c, g : c' \to c, \alpha : f \to g \vdash^{\whisk}_{c:\Ob;\alpha} \langle c' \rangle : (c : \Ob)}
            \infer[no rule]1{&c' : \Ob \vdash^{\texttt{uni}}_\emptyset \langle \top \rangle : (c : \Ob)}
            \infer1{&\multiquad[3]c' : \Ob \vdash^{\texttt{uni}}_\emptyset \langle \top, c' \rangle : (c : \Ob, c' : \Ob) \multiquad[3]}
        \end{prooftree}
    \end{center}
    Three more applications give
    \begin{align*}
        c' : \Ob, f' : c' \to \top, g' : c' \to \top \vdash 
        \raisebox{-1ex}{\resizebox{1.2\width}{2.3\height}{\Bigg \langle}}
        \begin{matrix*}
            \top \\[1mm]
            c' \\[1mm]
            f'  \\[1mm]
            g ' \\[1mm]
            \texttt{uni}_{\emptyset,\alpha}
        \end{matrix*}
    \raisebox{-1ex}{\resizebox{1.2\width}{2.3\height}{\Bigg \rangle}} :
    \begin{pmatrix*}
        c : \Ob \\[1mm]
        c' : \Ob \\[1mm]
        f : c' \to c \\[1mm]
        g : c' \to c \\[1mm]
        \alpha : f \to g
    \end{pmatrix*}
    \end{align*}
    Pictorially we have
    \begin{center}
        \begin{tikzcd}[column sep=large]
            {\top} & {c'}
            \arrow[""{name=0, anchor=center, inner sep=0}, "{f'}"', bend right=40, from=1-2, to=1-1]
            \arrow[""{name=1, anchor=center, inner sep=0}, "{g'}", bend left=40, from=1-2, to=1-1]
            \arrow["{\texttt{uni}_{\emptyset,\alpha}}", shorten <=3pt, shorten >=3pt, Rightarrow, from=0, to=1]
        \end{tikzcd}
    \end{center}
\end{example}

%%%%%%%%%%%%%%%%%%%%%%%%%%%%%%%%%%%%%%%%%%%%%%%%%%%%%
%%%%%%%%%%%%%%%%%%%%%%%%%%%%%%%%%%%%%%%%%%%%%%%%%%%%%
\subsection{Free variables and cut admissibility} \label{subsec:FreeVars}
%%%%%%%%%%%%%%%%%%%%%%%%%%%%%%%%%%%%%%%%%%%%%%%%%%%%%
%%%%%%%%%%%%%%%%%%%%%%%%%%%%%%%%%%%%%%%%%%%%%%%%%%%%%

For the type theory to function properly we need to make sure the cut rule is admissible,~i.e. we can perform substitution. To ensure the admissibility of the cut rule we do the usual trick: introduce just enough cut into the remaining rules. In our case, it suffices to do this for all the rules introducing new terms. For the universal cone, the rules take the form
\medskip
\begin{center}
    \begin{prooftree}
        \hypo{(\Theta,x:X,\Theta') \textup{\texttt{ cone }} (\Gamma;c)}
        \hypo{\Gamma \vdash^{\textup{\texttt{uni}}}_\Gamma \kappa : \Theta}
        \hypo{\Delta \vdash \gamma : \Gamma}
        \infer3{\Delta \vdash \textup{\texttt{ucone}}_{\Gamma,x}[\gamma] : X[\kappa \circ \gamma]}
    \end{prooftree}
\end{center}
\bigskip
\begin{center}
    \begin{prooftree}
        \hypo{(\Theta,x:X,\Theta') \textup{\texttt{ cone }} (\Gamma;c)}
        \hypo{\Gamma \vdash^{\textup{\texttt{uni}}}_\Gamma \kappa : \Theta}
        \infer2{\Gamma \vdash^{\textup{\texttt{uni}}}_\Gamma \langle \kappa, \textup{\texttt{ucone}}_{\Gamma,x}[\id_\Gamma] \rangle : (\Theta, x : X)}
    \end{prooftree}
\end{center}
\medskip
For the universal property we make similar modifications. For $ J_1 $ and $ J_2 $ rule we add in a cut and modify the rule accordingly
\medskip
\begin{center}
    \adjustbox{scale=0.87}{
        \begin{prooftree}
            \hypo{& K \textup{\texttt{ cone }} (\Gamma;c)}
            \infer[no rule]1{& M \textup{\texttt{ ctrf }}(\Gamma, c)}
            \hypo{&\Lambda, \alpha : T_\alpha, \Delta \vdash_{K, \Theta} x : X}
            \infer[no rule]1{&\Lambda, \alpha : T_\alpha, \Delta \vdash^\whisk_{K;\alpha } w : M}
            \hypo{\Omega \vdash^{\textup{\texttt{uni}}}_\Gamma \theta : (K,\Theta)}
            \hypo{\Phi \vdash \omega : \Omega}
            \infer4{J_1 \qquad\qquad J_2}
        \end{prooftree}
        \quad
        \begin{tabular}{c}
            \small $ \dim(\alpha) = n_M $
        \end{tabular}
    }
\end{center}
where
\begin{align*}
    J_1: && &\Phi \vdash \textup{\texttt{uni}}_{\Gamma,x}[\omega] : X[\theta\circ \omega] && x \in \{\alpha\} \\[2mm]
    J_2: && &\Omega \vDash \textup{\texttt{uni}}_{\Gamma,x}[\id_\Gamma] : X[\theta] && x \in \FV(\Delta).
\end{align*}
\bigskip
The rule for $ J_3 $ remains unchanged, while for $ J_4 $ the premise of the rules stays the same but the conclusion becomes
\begin{align*}
    J_4: && &\Omega \vdash^{\textup{\texttt{uni}}}_\Gamma \langle \theta, \textup{\texttt{uni}}_{\Gamma,x}[\id_\Gamma] \rangle : (\Theta, x : X) && x \in \{\alpha\} \cup \FV(\Delta).
\end{align*}
Finally, we add a cut to the rule producing the coinductive inverses:
\medskip
\begin{center}
    \begin{prooftree}
        \hypo{\Gamma \vDash u : s \to_A t}
        \hypo{\Delta \vdash \gamma : \Gamma}
        \infer2{\Delta \vDash \texttt{inv}(u[\gamma]) : t[\gamma] \to_{A[\gamma]} s[\gamma]}
    \end{prooftree}
\end{center}
\medskip
\begin{center}
    \begin{prooftree}
        \hypo{\Gamma \vDash u : s \to_A t}
        \hypo{\Delta \vdash \gamma : \Gamma}
        \infer2{\Delta \vDash \texttt{eta}(u[\gamma]) : 1_{s[\gamma]} \to u[\gamma] \cdot \texttt{inv}(u[\gamma]) \qquad \Delta \vDash \texttt{eps}(u[\gamma]) : \texttt{inv}(u[\gamma]) \cdot u[\gamma]  \to 1_{t[\gamma]}}
    \end{prooftree}
\end{center}
\medskip

To make substitution admissible we then make the following definition.
\begin{definition} Substitution on term constructors is defined by    
    \begin{align*}
        \textup{\texttt{ucone}}_{\Gamma,x}[\gamma][\delta] &\equiv \textup{\texttt{ucone}}_{\Gamma,x}[\gamma \circ \delta], \qquad & \textup{\texttt{inv}}(u)[\gamma] &\equiv \textup{\texttt{inv}}(u[\gamma]) \\
        \textup{\texttt{uni}}_{\Gamma,x}[\gamma][\delta] &\equiv \textup{\texttt{uni}}_{\Gamma,x}[\gamma \circ \delta], & \textup{\texttt{eta}}(u)[\gamma] &\equiv \textup{\texttt{eta}}(u[\gamma]) \\
        & &
        \textup{\texttt{eps}}(u)[\gamma] &\equiv \textup{\texttt{eps}}(u[\gamma]).
    \end{align*}
\end{definition}
With this at hand one can prove:
\begin{lemma} 
	The following rules are admissible in \textup{\texttt{CaTT}}
	\vspace{1mm}
	\begin{enumerate}[label=(\roman*)]
		\setlength\itemsep{0.7em}
		\item For types: \quad
		
			\begin{prooftree}
				\hypo{\Gamma \vdash A}
				\hypo{\Delta \vdash \gamma : \Gamma}
				\infer2{\Delta \vdash A[\gamma]}
			\end{prooftree}
		\quad \ and \quad $ A[\gamma][\delta] \equiv A[\gamma \circ \delta] $ \quad for all \quad $ \Phi \vdash \delta : \Delta $.
	
		\item For terms: \quad
		
			\begin{prooftree}
				\hypo{\Gamma \vdash t : A}
				\hypo{\Delta \vdash \gamma : \Gamma}
				\infer2{\Delta \vdash t[\gamma] : A[\gamma]}
			\end{prooftree}
		\quad and \quad $ t[\gamma][\delta] \equiv t[\gamma \circ \delta] $ \quad for all \quad $ \Phi \vdash \delta : \Delta $,

        and the same rule holds with $ \vdash $ replaced with $ \vDash $.

		\item For contexts: \quad
		
			\begin{prooftree}
				\hypo{\Gamma \vdash \theta : \Theta}
				\hypo{\Delta \vdash \gamma : \Gamma}
				\infer2{\Delta \vdash \theta \circ \gamma : \Theta}
			\end{prooftree}
		\ \ and \ \ $ (\theta \circ \gamma) \circ \delta \equiv \theta \circ (\gamma \circ \delta) $ \ for all \ $ \Phi \vdash \delta : \Delta $.
	\end{enumerate}
\end{lemma}

Finally we give the definition of the free variables of the new term constructors.
\begin{definition}
    The free variables of the term constructors are defined by
    \begin{align*}
        \FV(\textup{\texttt{ucone}}_{\Gamma,x}[\gamma]) &:= \FV(\gamma) & \FV(\textup{\texttt{inv}}(u)[\gamma]) &:= \FV(u[\gamma]) \\
        \FV(\textup{\texttt{uni}}_{\Gamma,x}[\omega]) &:= \FV(\omega) & \FV(\textup{\texttt{eta}}(u)[\gamma]) &:= \FV(u[\gamma]) \\
        & & \FV(\textup{\texttt{eps}}(u)[\gamma]) &:= \FV(u[\gamma]).
    \end{align*}
\end{definition}

%%%%%%%%%%%%%%%%%%%%%%%%%%%

In total, we define a type theory $ \textup{\texttt{CaTT}}_\textup{{\texttt{lim}}}$, which extends \CaTT{} and describes $(\infty,\infty)$-categories with lax limits for finite computads.

\begin{definition}
    The type theory $ \textup{\texttt{CaTT}}_\textup{{\texttt{lim}}}$ given by the rules of \CaTT{}, together with the rules:

    Term constructor for the universal cone:

    \begin{center}
        \begin{prooftree}
            \hypo{(\Theta,x:X,\Theta') \textup{\texttt{ cone }} (\Gamma;c)}
            \hypo{\Gamma \vdash^{\textup{\texttt{uni}}}_\Gamma \kappa : \Theta}
            \hypo{\Delta \vdash \gamma : \Gamma}
            \infer3{\Delta \vdash \textup{\texttt{ucone}}_{\Gamma,x}[\gamma] : X[\kappa \circ \gamma]}
        \end{prooftree}
    \end{center}
    \bigskip

    Term constructors for the universal property:

    \begin{center}
        \adjustbox{scale=0.87}{
            \begin{prooftree}
                \hypo{& K \textup{\texttt{ cone }} (\Gamma;c)}
                \infer[no rule]1{& M \textup{\texttt{ ctrf }}(\Gamma, c)}
                \hypo{&\Lambda, \alpha : T_\alpha, \Delta \vdash_{K, \Theta} x : X}
                \infer[no rule]1{&\Lambda, \alpha : T_\alpha, \Delta \vdash^\whisk_{K;\alpha } w : M}
                \hypo{\Omega \vdash^{\textup{\texttt{uni}}}_\Gamma \theta : (K,\Theta)}
                \hypo{\Phi \vdash \omega : \Omega}
                \infer4{J_1 \qquad\qquad J_2}
            \end{prooftree}
            \quad
            \begin{tabular}{c}
                \small $ \dim(\alpha) = n_M $
            \end{tabular}
        }
    \end{center}
    where
    \begin{align*}
        J_1: && &\Phi \vdash \textup{\texttt{uni}}_{\Gamma,x}[\omega] : X[\theta\circ \omega] && x \in \{\alpha\} \\[2mm]
        J_2: && &\Omega \vDash \textup{\texttt{uni}}_{\Gamma,x}[\id_\Gamma] : X[\theta] && x \in \FV(\Delta).
    \end{align*}

    Term constructors for invertible morphisms:

    \medskip
    \begin{center}
        \begin{prooftree}
            \hypo{\Gamma \vDash u : s\to_A t}
            \infer1{\Gamma \vdash u : s \to_A t}
        \end{prooftree}
    \end{center}
    \bigskip
    \begin{center}
        \begin{prooftree}
            \hypo{\Gamma \vDash u : s \to_A t}
            \hypo{\Delta \vdash \gamma : \Gamma}
            \infer2{\Delta \vDash \textup{\texttt{inv}}(u[\gamma]) : t[\gamma] \to_{A[\gamma]} s[\gamma]}
        \end{prooftree}
    \end{center}
    \bigskip
    \begin{center}
        \begin{prooftree}
            \hypo{\Gamma \vDash u : s \to_A t}
            \hypo{\Delta \vdash \gamma : \Gamma}
            \infer2{\Delta \vDash \textup{\texttt{eta}}(u[\gamma]) : 1_{s[\gamma]} \to u[\gamma] \cdot \textup{\texttt{inv}}(u[\gamma]) \qquad \Delta \vDash \textup{\texttt{eps}}(u[\gamma]) : \textup{\texttt{inv}}(u[\gamma]) \cdot u[\gamma]  \to 1_{t[\gamma]}}
        \end{prooftree}
    \end{center}
    \bigskip
These rules are accompanied by a set of recognition rules for special contexts and context morphisms, listed below. 

Rules for cones as contexts:
    \bigskip
    \begin{center}
        \begin{prooftree}
            \infer0[\textup{\footnotesize{(EK)}}]{ c : \Ob \ \cone \ (\emptyset, c)}
        \end{prooftree}
    \end{center}
    \medskip
    \begin{center}
        \adjustbox{scale=0.95}{
        \begin{prooftree}
            \hypo{\Gamma, c : \Ob, \Pi \ \cone \ (\Gamma;c)}
            \hypo{\Gamma, x : X, c : \Ob, \Pi \vdash s \to_A t}
            \infer2[\textup{\footnotesize{(KE)}}]{\Gamma, x : X, c : \Ob, \Pi, p_x : s \to_A t \ \cone \ ((\Gamma,x : X),c)}
        \end{prooftree}
        \quad \footnotesize
        \begin{tabular}{l}
            $ \tau\overline{\textup{Cond}}(s : A, x : X, c : \Ob, \Pi) $ \\[1mm]
            $\sigma\textup{Cond}(t : A, x : X, \Pi) $
        \end{tabular} \normalsize
        }
    \end{center}
    \bigskip
    where for $ \delta \in \{\sigma,\tau\}$ the notation $ \delta{\textup{Cond}}(s : A, x : X, \Pi) $ stands for $ t : A $ being categorical, $ x \propto t $ and
        $$ 
        \FV(t : A) = \FV(x : X) \cup \bigcup\limits_{ \substack{p \, \in \, \FV(\Pi) \\[1mm] y \, \in \, \FV(\delta(x) \, : \, \partial X) \\[1mm] y \, \propto \, \tau(p) } } \FV(p : T_p) 
        $$
    while $ \tau\overline{\textup{Cond}}(s : A, x : X, c : \Ob, \Pi) $ stands for $ t : A $ being categorical and 
        $$ 
        \FV(t : A) = \FV(c : \Ob) \cup \bigcup\limits_{ \substack{p \, \in \, \FV(\Pi) \\[1mm] y \, \in \, \FV(\delta(x) \, : \, \partial X) \\[1mm] y \, \propto \, \tau(p) } } \FV(p : T_p) .
        $$

    \medskip
    Rules for higher conical transfors as contexts:

        \bigskip
    \begin{center}
        \begin{prooftree}
            \infer0{\emptyset;c:\Ob; \emptyset;\dots; \emptyset \ \textup{\texttt{ctrf}} \ (\emptyset,c)}
        \end{prooftree}
    \end{center}
    \bigskip
    \begin{center}
        \begin{prooftree}
            \hypo{\Gamma; c : \Ob; M_1; \dots; M_{2n+1} \ \textup{\texttt{ctrf}} \ (\Gamma;c)}
            \hypo{\Gamma, x : X, c : \Ob, M_1, p : T \ \textup{\texttt{cone}} \ ((\Gamma,x :X),c)}
            \infer2{\Gamma,x :X; c : \Ob; M_1, p : T; M_2; \dots; M_{2n+1} \ \textup{\texttt{pctrf}} \ ((\Gamma,x:X),c)}
        \end{prooftree}
    \end{center}
    \bigskip
    \begin{center}
        \adjustbox{scale=0.8}{
        \begin{prooftree}
            \hypo{&\Gamma; c : \Ob; M_1; \dots ; M_{2i-1}, x : X; M_{2i}; \dots ; M_{2n+1} \ \textup{\texttt{pctrf}} \ (\Gamma;c)}
            \infer[no rule]1{&\Gamma; c : \Ob; M_1; \dots; M_{2i-1}, x : X; M_{2i}, x' : X; M_{2i+1} \vdash s \to_A t}
            \infer1{&\Gamma; c :\Ob; M_1; \dots ; M_{2i-1}, x : X; M_{2i}, x' : X; M_{2i+1}, p_x : s \to_A t; \dots ; M_{2n+1} \ \textup{\texttt{pctrf}} \ (\Gamma;c)}
        \end{prooftree}
        \ \footnotesize \begin{tabular}{c}
            $ 1 \le i \le n $ \\[1mm]
            $ |M_{2i-1}|=|M_{2i}| $ \\[1mm]
            $ \displaystyle{\tau\textup{Cond}\Big(s : A, x : X, \FV(M_{2i+1}) \Big) }  $ \\[1mm]
            $ \displaystyle{\sigma\textup{Cond}\Big(t : A, x' : X, \FV(M_{2i+1})\Big) }  $
        \end{tabular}
        } \normalsize
    \end{center}
    \bigskip
    \begin{center}
        \begin{prooftree}
            \hypo{\Gamma; c : \Ob; M_1; \dots; M_{2n+1} \ \textup{\texttt{pctrf}} \ (\Gamma;c)}
            \infer1{\Gamma; c : \Ob; M_1; \dots; M_{2n+1} \ \textup{\texttt{ctrf}} \ (\Gamma;c)}
        \end{prooftree}
        \qquad \footnotesize $ |M_1| = |M_{2n+1}| $ \normalsize
    \end{center}
    \bigskip

    Rules for postcomposing with a cone:

    \begin{center}
        \adjustbox{scale=0.8}{
        \begin{prooftree}
            \hypo{K \textup{\texttt{ cone }} (\Gamma, c)}
            \infer1{K, c' : \Ob, D^{n}(c,c',\alpha) \vdash^\textup{\pstar}_{K; \alpha } \langle \id_\Gamma, c' \rangle : (\Gamma, c : \Ob)}
        \end{prooftree}
        }
    \end{center}
    \bigskip
    \begin{center}
        \adjustbox{scale=0.8}{
        \begin{prooftree}
            \hypo{K \textup{\texttt{ cone }} (\Gamma, c)}
            \infer[no rule]1{M \textup{\texttt{ ctrf }} (\Gamma, c)}
            \hypo{W \vdash^\textup{\pstar}_{K;\alpha } w : \Theta}
            \hypo{M \vdash_\Theta x : X}
            \hypo{\Theta \vdash u : X[w]}
            \infer4{W \vdash^\textup{\pstar}_{K; \alpha } \langle w, u \rangle : (\Theta, x : X)}
        \end{prooftree}
        \quad
        \begin{tabular}{c}
            $\dim(\alpha) = n_M $\\[2mm]
            $x \in \FV(M_{j}), \ j \in \{1, \dots, 2n_M - 2, 2n_M \} $ \\[2mm]
            $ \textup{Star}(\Gamma, K, M, x, u : X[w], \alpha : A) $
        \end{tabular}
        }
    \end{center}
    \bigskip
    \begin{center}
        \adjustbox{scale=0.8}{
        \begin{prooftree}
            \hypo{K \textup{\texttt{ cone }} (\Gamma, c)}
            \infer[no rule]1{M \textup{\texttt{ ctrf }} (\Gamma, c)}
            \hypo{W \vdash^\textup{\pstar}_{K;\alpha } w : \Theta}
            \hypo{M \vdash_\Theta x : X}
            \infer3{W, x' : X[w] \vdash^\textup{\pstar}_{K;\alpha } \langle w, x' \rangle : (\Theta, x : X)}
        \end{prooftree}
        \quad
        \begin{tabular}{c}
            $\dim(\alpha) = n_M $\\[2mm]
            $x \in \FV(M_{j}), \ j \in \{ 2n_M - 1, 2n_M +1 \} $
        \end{tabular}
        }
    \end{center}
    \bigskip
    \begin{center}
        \adjustbox{scale=0.8}{
        \begin{prooftree}
            \hypo{W, \alpha : A, x : X, W' \vdash^\textup{\pstar}_{K,\alpha } w : \Theta}
            \infer1{W, x : X, \alpha : A, W' \vdash^\textup{\pstar}_{K;\alpha } w : \Theta}
        \end{prooftree}
        \quad
        \begin{tabular}{c}
            $ \alpha \not \in \FV(X) $
        \end{tabular}
        }
    \end{center}
    \bigskip
    \begin{center}
        % \qquad\quad
        \adjustbox{scale=0.8}{
        \begin{prooftree}
            \hypo{W, \alpha : A, x : X, W' \vdash^\textup{\pstar}_{K,\alpha } w : \Theta}
            \infer1{W, \alpha : A, x : X, W' \vdash^\textup{\whisk}_{K;\alpha } w : \Theta}
        \end{prooftree}
        \quad
        \begin{tabular}{c}
            $ \alpha \in \FV(X) $
        \end{tabular}
        }
    \end{center}
    \bigskip
    where, with the notation $ \Gamma \equiv (x_i : A_i)_{1 \le i \le l} $ and $ K \equiv (\Gamma, c : \Ob, (p_i : T_{x_i})_{1\le i \le l}) $, as well as $ M_{j} \equiv (x_{j,i} : T_{j,i})_{1 \le i \le l} $, the shorthand $ \textup{Star}(\Gamma,K,M,x_{j,i}, u : X, \alpha : T_\alpha) $ stands for
    \begin{align*}
        \FV(u : X) &= \FV(\alpha : T_\alpha) \cup \FV(p_i : T_i) \\[3mm]
        \FV(u : X) &= \begin{cases}
            \FV(\sigma^{n-j}(\alpha) : \partial T_\alpha ) \cup \FV(p_i : T_{x_i}), \ & i \text{ odd} \\
            \FV(\tau^{n-j}(\alpha) : \partial T_\alpha )\cup \FV(p_i : T_{x_i}), \ & i \text{ even}
        \end{cases}
    \end{align*}

    Rules for context morphisms for the universal cone and the universal property: 
    \begin{center}
        \begin{prooftree}
            \hypo{\Gamma \vdash}
            \infer1{\Gamma \vdash^{\text{\textup{uni}}}_\Gamma \id_\Gamma : \Gamma}
        \end{prooftree}
    \end{center}
    \bigskip
    \begin{center}
        \begin{prooftree}
            \hypo{(\Theta,x:X,\Theta') \textup{\texttt{ cone }} (\Gamma;c)}
            \hypo{\Gamma \vdash^{\textup{\texttt{uni}}}_\Gamma \kappa : \Theta}
            \infer2{\Gamma \vdash^{\textup{\texttt{uni}}}_\Gamma \langle \kappa, \textup{\texttt{ucone}}_{\Gamma,x}[\id_\Gamma] \rangle : (\Theta, x : X)}
        \end{prooftree}
    \end{center}

    \bigskip

    \begin{center}
        \begin{prooftree}
            \hypo{& K \textup{\texttt{ cone }} (\Gamma;c)}
            \infer[no rule]1{& M \textup{\texttt{ ctrf }}(\Gamma, c)}
            \hypo{&\Lambda, \alpha : T_\alpha, \Delta \vdash_{K, \Theta} x : X}
            \infer[no rule]1{&\Lambda, \alpha : T_\alpha, \Delta \vdash^\whisk_{K;\alpha } w : M}
            \hypo{\Omega \vdash^{\textup{\texttt{uni}}}_\Gamma \theta : (K,\Theta)}
            \infer3{J_3 \qquad\qquad J_4}
        \end{prooftree}
    \end{center}
    
    \bigskip

    where
    \begin{align*}
        \displaystyle
        J_3: && &\Omega, x' : X[\theta] \vdash^{\textup{\texttt{uni}}}_\Gamma \langle \theta, x' \rangle : (\Omega, x : X) && x \in \FV(\Lambda) \\[2mm]
        \displaystyle J_4: && &\Omega \vdash^{\textup{\texttt{uni}}}_\Gamma \langle \theta, \textup{\texttt{uni}}_{\Gamma,x}[\id_\Gamma] \rangle : (\Theta, x : X) && x \in \{\alpha\} \cup \FV(\Delta).
    \end{align*}

\end{definition}

% \bibliographystyle{acm} 
% \bibliography{Bibliography}

\end{document}